\documentclass[12pt]{amsart}
\usepackage{amscd,amsmath,amssymb,enumerate,amsfonts,mathrsfs,txfonts}
\usepackage{comment}
\usepackage{hyperref}
\usepackage{xcolor}
\usepackage{tikz,pgfplots,pgf}
\usepackage{vmargin}
\usepackage[all]{xy}

\tikzset{node distance=3cm, auto}
\usetikzlibrary{calc} 
\usetikzlibrary{trees} 
\setpapersize{A4}
\setmargins{2.5cm}      
{1.5cm}                        
{16.5cm}                      
{23.42cm}                   
{10pt}                          
{1cm}                           
{0pt}                          
{2cm}

\newtheorem{theorem}{Theorem}[section]
\theoremstyle{plain}

\newtheorem{corollary}[theorem]{Corollary}

\newtheorem{definition}[theorem]{Definition}

\newtheorem{lemma}[theorem]{Lemma}

\newtheorem{proposition}[theorem]{Proposition}

\numberwithin{equation}{section}

\def\N{\mathbb{N}}
\def\R{\mathbb{R}}
\def\C{\mathbb{C}}

\def\G{\mathcal{G}}
\def\H{\mathcal{H}}
\def\I{\mathcal{I}}
\def\L{\mathcal{L}}
\def\D{\mathbb{D}}
\def\B{\mathcal{B}}
\def\P{\mathcal{P}}

\def\Aut{\mathrm{Aut}}
\def\A{\mathcal{A}}

\def\lin{\mathrm{lin}}

\def\T{\mathbb{T}}
\def\rang{\mathrm{rang}}

\makeatletter
\@namedef{subjclassname@2020}{\textup{2020} Mathematics Subject Classification}
\makeatother

\begin{document}

\title[$p$-Summing Bloch mappings on the complex unit disc]{$p$-Summing Bloch mappings on the complex unit disc}

\author[M. G. Cabrera-Padilla]{M. G. Cabrera-Padilla}
\address[]{Departamento de Matem\'{a}ticas, Universidad de Almer\'{i}a, Ctra. de Sacramento s/n, 04120 La Ca\~nada de San Urbano, Almer\'{i}a, Spain}
\email{m\_gador@hotmail.com}
\thanks{Research partially supported by Junta de Andaluc\'{\i}a grant FQM194. The first two authors were supported by grant PID2021-122126NB-C31 funded by MCIN/AEI/ 10.13039/501100011033 and by ``ERDF A way of making Europe''.}

\author[A. Jim{\'e}nez-Vargas]{A. Jim\'enez-Vargas}
\address[]{Departamento de Matem\'aticas, Universidad de Almer\'ia, Ctra. de Sacramento s/n, 04120 La Ca\~nada de San Urbano, Almer\'ia, Spain}
\email{ajimenez@ual.es}

\author[D. Ruiz-Casternado]{D. Ruiz-Casternado}
\address[]{Departamento de Matem{\'a}ticas, Universidad de Almer{\'i}a, Ctra. de Sacramento s/n, 04120 La Ca\~nada de San Urbano, Almer{\'i}a, Spain}
\email{drc446@ual.es}




\begin{abstract}
The notion of $p$-summing Bloch mapping from the complex unit open disc $\D$ into a complex Banach space $X$ is introduced for any $1\leq p\leq\infty$. It is shown that the linear space of such mappings,  equipped with a natural seminorm $\pi^{\B}_p$, is M\"obius-invariant. Moreover, its subspace consisting of all those mappings which preserve the zero is an injective Banach ideal of normalized Bloch mappings. Bloch versions of the Pietsch's domination/factorization Theorem and the Maurey's extrapolation Theorem are presented. We also introduce the spaces of $X$-valued Bloch molecules on $\D$ and identify the spaces of normalized $p$-summing Bloch mappings from $\D$ into $X^*$ under the norm $\pi^{\B}_p$ with the duals of such spaces of molecules under the Bloch version of the $p^*$-Chevet--Saphar tensor norms $d_{p^*}$.
\end{abstract}
\maketitle


\section*{Introduction}

The known concept of absolutely $p$-summing operator between Banach spaces, introduced by Grothendieck \cite{Gro-55} for $p=1$ and by Pietsch \cite{Pie-67} for any $p>0$, can be adapted to address the property of summability in the setting of Bloch mappings from the complex unit open disc $\D$ into a complex Banach space $X$ as follows.

The study of summability has been addressed for different classes of mappings by some authors. For example, for multilinear operators by Achour and Mezrag \cite{AchMez-07} and Dimant \cite{Dim-03}, for Lipschitz mappings by Farmer and Johnson \cite{FarJoh-09} and Saadi \cite{Saa-15}, and for holomorphic mappings by Matos \cite{Mat-96} and Pellegrino \cite{Pel-03}, among other settings. See also the survey by Pellegrino, Rueda and S\'anchez-P\'erez \cite{PelRuSan-16} for the summability on multilinear operators and homogeneous polynomials.

If $\H(\D,X)$ denotes the space of all holomorphic mappings from $\D$ into $X$, let us recall that a mapping $f\in\H(\D,X)$ is called \emph{Bloch} if there exists a constant $c\geq 0$ such that $(1-|z|^2)\left\|f'(z)\right\|\leq c$ for all $z\in\D$. 

The \textit{Bloch space} $\B(\D,X)$ is the linear space of all those mappings $f\in\H(\D,X)$ such that   
$$
p_{\B}(f):=\sup\left\{(1-|z|^2)\left\|f'(z)\right\|\colon z\in\D\right\}<\infty ,
$$ 
equipped with the \textit{Bloch seminorm} $p_{\B}$. The \textit{normalized Bloch space} $\widehat{\B}(\D,X)$ is the Banach space of all Bloch mappings from $\D$ into $X$ such that $f(0)=0$, equipped with the \textit{Bloch norm} $p_{\B}$. In particular, we will write $\widehat{\B}(\D)$ instead of $\widehat{\B}(\D,\C)$. We refer the reader to \cite{AndCluPom-74,Zhu-07} for the scalar-valued theory, and to \cite{ArrBlas-03, Bla-90} for the vector-valued theory on these spaces. 

For any $1\leq p\leq\infty$, we say that a mapping $f\in\H(\D,X)$ is \textit{$p$-summing Bloch} if there is a constant $c\geq 0$ such that for any $n\in\mathbb{N}$, $\lambda_1,\ldots,\lambda_n\in\mathbb{C}$ and $z_1,\ldots,z_n\in\D$, we have 
\begin{align*}
\left(\sum_{i=1}^n\left|\lambda_i\right|^p\left\|f'(z_i)\right\|^p\right)^{\frac{1}{p}}&\leq c
\sup_{g\in B_{\widehat{\B}(\D)}}\left(\sum_{i=1}^n\left|\lambda_i\right|^p\left|g'(z_i)\right|^p\right)^{\frac{1}{p}}\quad & \text{if}\quad& 1\leq p<\infty, \\
\max_{1\leq i\leq n}\left|\lambda_i\right|\left\|f'(z_i)\right\|&\leq c \sup_{g\in B_{\widehat{\B}(\D)}}\left(\max_{1\leq i\leq n}\left|\lambda_i\right|\left|g'(z_i)\right|\right) \quad & \text{if}\quad &p=\infty.
\end{align*}
The infimum of all the constants $c$ for which such an inequality holds, denoted $\pi^{\B}_p(f)$, defines a seminorm on the linear space, denoted $\Pi^{\B}_p(\D,X)$, of all $p$-summing Bloch mappings $f\colon\D\to X$. Furthermore, this seminorm becomes a norm on the subspace $\Pi^{\widehat{\B}}_p(\D,X)$ consisting of all those mappings $f\in\Pi^{\B}_p(\D,X)$ so that $f(0)=0$.

These spaces enjoy nice properties in both complex and functional analytical frameworks. In the former setting, we show that the space $(\Pi^{\B}_p(\D,X),\pi^{\B}_p)$ is invariant by M\"obius transformations of $\D$. In the latter context and in a clear parallelism with the theory of absolutely $p$-summing linear operators (see \cite[Chapter 2]{DisJarTon-95}), we prove that $[\Pi^{\widehat{\B}}_p,\pi^{\B}_p]$ is an injective Banach ideal of normalized Bloch mappings whose elements can be characterized by means of Pietsch domination/factorization. Applying this Pietsch domination, we present a Bloch version of Maurey's extrapolation Theorem \cite{Mau-74}. 

On the other hand, the known duality of the Bloch spaces (see \cite{AndCluPom-74,ArrBlas-03,Zhu-07}) is extended to the spaces $(\Pi^{\widehat{\B}}_p(\D,X^*),\pi^{\B}_p)$ by identifying them with the duals of the spaces of the so-called \emph{$X$-valued Bloch molecules on $\D$}, equipped with the Bloch versions of the $p^*$-Chevet--Saphar tensor norms $d_{p^*}$. We conclude the paper with some open problems.

The proofs of some of our results are similar to those of their corresponding linear versions, but a detailed reading of them shows that the adaptation of the linear techniques to the Bloch setting is far from being simple. Our approach depends mainly on the application of some concepts and results concerning the theory on a strongly unique predual of the space $\widehat{\B}(\D)$, called \textit{Bloch-free Banach space over $\D$} that was introduced in \cite{JimRui-22}.\\

\textbf{Notation.} For two normed spaces $X$ and $Y$, $\mathcal{L}(X,Y)$ denotes the normed space of all bounded linear operators from $X$ to $Y$, equipped with the operator canonical norm. In particular, the topological dual space $\mathcal{L}(X,\mathbb{C})$ is denoted by $X^*$. For $x\in X$ and $x^*\in X^*$, we will sometimes write $\langle x^*,x\rangle=x^*(x)$. As usual, $B_X$ and $S_X$ stand for the closed unit ball of $X$ and the unit sphere of $X$, respectively. 
Let $\T$ and $\D$ denote the unit sphere and the unit open disc of $\C$, respectively. 

Given $1\leq p\leq\infty$, let $p^*$ denote the \textit{conjugate index of $p$} defined by 
$$
p^*=\left\{
\begin{array}{cll}
\infty & \text{if} & p=1, \\ 
p/(p-1) & \text{if} & 1<p<\infty, \\ 
1 & \text{if} & p=\infty .
\end{array}%
\right. 
$$


\section{$p$-Summing Bloch mappings on the unit disc}\label{2}

This section gathers the most important properties of $p$-summing Bloch mappings on $\D$. From now on, unless otherwise stated, $X$ will denote a complex Banach space.


\subsection{Inclusions}

We will first establish some useful inclusion relations. See first \cite[Satz 5]{Pie-67}.

The following class of Bloch functions will be used throughout the paper. For each $z\in\D$, the function $f_z\colon\D\to\C$ defined by 
$$
f_z(w)=\frac{(1-|z|^2)w}{1-\overline{z}w}\qquad (w\in\D),
$$ 
belongs to $\widehat{\B}(\D)$ with $p_{\B}(f_z)=1=(1-|z|^2)f_z'(z)$ (see \cite[Proposition 2.2]{JimRui-22}).

\begin{proposition}\label{prop-1}
Let $1\leq p<q\leq\infty$. Then $\Pi^{\B}_p(\D,X)\subseteq\Pi^{\B}_q(\D,X)$ with $\pi^{\B}_q(f)\leq\pi^{\B}_p(f)$ for all $f\in\Pi^{\B}_p(\D,X)$. Moreover, $\Pi^{\B}_\infty(\D,X)=\B(\D,X)$ with $\pi^{\B}_\infty(f)=p_\B(f)$ for all $f\in\Pi^{\B}_\infty(\D,X)$.
\end{proposition}

\begin{proof}
Let $n\in\mathbb{N}$, $\lambda_1,\ldots,\lambda_n\in\mathbb{C}$ and $z_1,\ldots,z_n\in\D$. We will first prove the second assertion. Let $f\in\Pi^{\B}_\infty(\D,X)$. For all $z\in\D$, we have
$$
(1-|z|^2)\left\|f'(z)\right\|\leq\pi^{\B}_\infty(f)\sup_{g\in B_{\widehat{\B}(\D)}}(1-|z|^2)\left|g'(z)\right|=\pi^{\B}_\infty(f),
$$
hence $f\in\B(\D,X)$ with $p_\B(f)\leq\pi^{\B}_\infty(f)$. Conversely, let $f\in\B(\D,X)$. For $i=1,\ldots,n$, we have 
$$
\left|\lambda_i\right|\left\|f'(z_i)\right\|\leq \frac{\left|\lambda_i\right|}{1-|z_i|^2}p_\B(f)=\left|\lambda_i\right|\left|f'_{z_i}(z_i)\right|p_\B(f)
\leq p_\B(f)\sup_{g\in B_{\widehat{\B}(\D)}}\left|\lambda_i\right|\left|g'(z_i)\right|,
$$
this implies that  
$$
\max_{1\leq i\leq n}\left|\lambda_i\right|\left\|f'(z_i)\right\|\leq p_\B(f)\sup_{g\in B_{\widehat{\B}(\D)}}\left(\max_{1\leq i\leq n}\left|\lambda_i\right|\left|g'(z_i)\right|\right),
$$
and thus $f\in\Pi^{\B}_\infty(\D,X)$ with $\pi^{\B}_\infty(f)\leq p_\B(f)$.

To prove the first assertion, let $f\in\Pi^{\B}_p(\D,X)$. Assume $q<\infty$. Taking $\beta_i=\left|\lambda_i\right|^{q/p}\left\|f'(z_i)\right\|^{(q/p)-1}$ for $i=1,\ldots,n$, we have 
$$
\left(\sum_{i=1}^n\left|\lambda_i\right|^q\left\|f'(z_i)\right\|^q\right)^{\frac{1}{p}}=\left(\sum_{i=1}^n\left|\beta_i\right|^p\left\|f'(z_i)\right\|^p\right)^{\frac{1}{p}}
\leq \pi^{\B}_p(f) \sup_{g\in B_{\widehat{\B}(\D)}}\left(\sum_{i=1}^n\left|\beta_i\right|^p\left|g'(z_i)\right|^p\right)^{\frac{1}{p}}.
$$
Since $q/p>1$ and $(q/p)^*=q/(q-p)$, H\"older Inequality yields
\begin{align*}
\sup_{g\in B_{\widehat{\B}(\D)}}\left(\sum_{i=1}^n\left|\beta_i\right|^p\left|g'(z_i)\right|^p\right)^{\frac{1}{p}}
&=\sup_{g\in B_{\widehat{\B}(\D)}}\left(\sum_{i=1}^n\left(\left|\lambda_i\right|\left\|f'(z_i)\right\|\right)^{q-p}\left(\left|\lambda_i\right|\left|g'(z_i)\right|\right)^p\right)^{\frac{1}{p}}\\
&\leq\left(\sum_{i=1}^n\left|\lambda_i\right|^q\left\|f'(z_i)\right\|^q\right)^{\frac{1}{p}-\frac{1}{q}}\sup_{g\in B_{\widehat{\B}(\D)}}\left(\sum_{i=1}^n\left|\lambda_i\right|^q\left|g'(z_i)\right|^q\right)^{\frac{1}{q}},
\end{align*}
and thus we obtain 
$$
\left(\sum_{i=1}^n\left|\lambda_i\right|^q\left\|f'(z_i)\right\|^q\right)^{\frac{1}{q}}\leq \pi^{\B}_p(f) \sup_{g\in B_{\widehat{\B}(\D)}}\left(\sum_{i=1}^n\left|\lambda_i\right|^q\left|g'(z_i)\right|^q\right)^{\frac{1}{q}}.
$$
This shows that $f\in\Pi^{\B}_q(\D,X)$ with $\pi^{\B}_q(f)\leq\pi^{\B}_p(f)$ if $q<\infty$. For the case $q=\infty$, note that  
$$
(1-|z|^2)\left\|f'(z)\right\|\leq\pi^{\B}_p(f)\sup_{g\in B_{\widehat{\B}(\D)}}(1-|z|^2)\left|g'(z)\right|=\pi^{\B}_p(f)
$$
for all $z\in\D$, and thus $f\in\B(\D,X)=\Pi^{\B}_\infty(\D,X)$ with $\pi^{\B}_\infty(f)=p_{\B}(f)\leq\pi^{\B}_p(f)$.
\end{proof}


\subsection{Injective Banach ideal property}


Let us recall (see \cite[Definition 5.11]{JimRui-22}) that a \textit{normalized Bloch ideal} is a subclass $\I^{\widehat{\B}}$ of the class of all normalized Bloch mappings $\widehat{\B}$ such that for every complex Banach space $X$, the components
$$
\I^{\widehat{\B}}(\D,X):=\I^{\widehat{\B}}\cap\widehat{\B}(\D,X),
$$
satisfy the following properties:
\begin{itemize}
\item[(I1)] $\I^{\widehat{\B}}(\D,X)$ is a linear subspace of $\widehat{\B}(\D,X)$,
\item[(I2)] For every $g\in\widehat{\B}(\D)$ and $x\in X$, the mapping $g \cdot x\colon z\mapsto g(z)x$ from $\D$ to $X$ is in $\I^{\widehat{\B}}(\D,X)$,
\item[(I3)] The \textit{ideal property}: if $f\in\I^{\widehat{\B}}(\D,X)$, $h\colon\D\to\D$ is a holomorphic function with $h(0)=0$ and $T\in\L(X,Y)$ where $Y$ is a complex Banach space, then $T\circ f\circ h$ belongs to $\I^{\widehat{\B}}(\D,Y)$.
\end{itemize}

A normalized Bloch ideal $\I^{\widehat{\B}}$ is said to be \textit{normed (Banach)} if there is a function $\|\cdot\|_{\I^{\widehat{\B}}}\colon\I^{\widehat{\B}}\to\mathbb{R}_0^+$ such that for every complex Banach space $X$, the following three conditions are satisfied:
\begin{itemize}
\item[(N1)] $(\I^{\widehat{\B}}(\D,X),\|\cdot\|_{\I^{\widehat{\B}}})$ is a normed (Banach) space with $p_\B(f)\leq\|f\|_{\I^{\widehat{\B}}}$ for all $f\in\I^{\widehat{\B}}(\D,X)$,
\item[(N2)] $\|g\cdot x\|_{\I^{\widehat{\B}}}=p_\B(g)\left\|x\right\|$ for all $g\in\widehat{\B}(\D)$ and $x\in X$,
\item[(N3)] If $h\colon\D\to\D$ is a holomorphic function with $h(0)=0$, $f\in\I^{\widehat{\B}}(\D,X)$ and $T\in\L(X,Y)$ where $Y$ is a complex Banach space, then $\|T\circ f\circ h\|_{\I^{\widehat{\B}}}\leq\|T\|\,\|f\|_{\I^{\widehat{\B}}}$.
\end{itemize}

A normed normalized Bloch ideal $[\I^{\widehat{\B}},\|\cdot\|_{\I^{\widehat{\B}}}]$ is said to be:
\begin{itemize}
\item[(I)] \textit{Injective} if for any mapping $f\in\widehat{\B}(\D,X)$, any complex Banach space $Y$ and any isometric linear embedding $\iota\colon X\to Y$, we have that $f\in\I^{\widehat{\B}}(\D,X)$ with $\left\|f\right\|_{\I^{\widehat{\B}}}=\left\|\iota\circ f\right\|_{\I^{\widehat{\B}}}$ whenever $\iota\circ f\in\I^{\widehat{\B}}(\D,Y)$.
\end{itemize}

We are now ready to establish the following result which can be compared to \cite[Satzs 1--4]{Pie-67}.

\begin{proposition}\label{ideal summing} 
$[\Pi^{\widehat{\B}}_p,\pi^{\B}_p]$ is an injective Banach normalized Bloch ideal for any $1\leq p\leq\infty$.
\end{proposition}

\begin{proof}
Note that $\Pi^{\widehat{\B}}_p(\D,X)\subseteq\widehat{\B}(\D,X)$ with $p_{\B}(f)\leq\pi^{\B}_p(f)$ for all $f\in\Pi^{\widehat{\B}}_p(\D,X)$ by Proposition \ref{prop-1}.

We will only prove the case $1<p<\infty$. The cases $p=1$ and $p=\infty$ follow similarly. Let $n\in\mathbb{N}$, $\lambda_1,\ldots,\lambda_n\in\mathbb{C}$ and $z_1,\ldots,z_n\in\D$. 

(N1) If $f\in\Pi^{\widehat{\B}}_p(\D,X)$ and $\pi^{\B}_p(f)=0$, then $p_{\B}(f)=0$, and so $f=0$. Given $f_1,f_2\in\Pi^{\widehat{\B}}_p(\D,X)$, we have 
\begin{align*}
\left(\sum_{i=1}^n\left|\lambda_i\right|^p\left\|(f_1+f_2)'(z_i)\right\|^p\right)^{\frac{1}{p}}&
\leq\left(\sum_{i=1}^n\left|\lambda_i\right|^p\left(\left\|f'_1(z_i)\right\|^p+\left\|f'_2(z_i)\right\|^p\right)\right)^{\frac{1}{p}}\\
&\leq\left(\sum_{i=1}^n\left|\lambda_i\right|^p\left\|f_1'(z_i)\right\|^p\right)^{\frac{1}{p}}
+\left(\sum_{i=1}^n\left|\lambda_i\right|^p\left\|f_2'(z_i)\right\|^p\right)^{\frac{1}{p}}\\
&\leq\left(\pi^{\B}_p(f_1)+\pi^{\B}_p(f_2)\right)\sup_{g\in B_{\widehat{\B}(\D)}}\left(\sum_{i=1}^n\left|\lambda_i\right|^p\left|g'(z_i)\right|^p\right)^{\frac{1}{p}},
\end{align*}
and therefore $f_1+f_2\in\Pi^{\widehat{\B}}_p(\D,X)$ with $\pi^{\B}_p(f_1+f_2)\leq\pi^{\B}_p(f_1)+\pi^{\B}_p(f_2)$.

Let $\lambda\in\mathbb{C}$ and $f\in\Pi^{\widehat{\B}}_p(\D,X)$. We have 
\begin{align*}
\left(\sum_{i=1}^n\left|\lambda_i\right|^p\left\|(\lambda f)'(z_i)\right\|^p\right)^{\frac{1}{p}}
&=\left|\lambda\right|\left(\sum_{i=1}^n\left|\lambda_i\right|^p\left\|f'(z_i)\right\|^p\right)^{\frac{1}{p}}\\
&\leq\left|\lambda\right|\pi^{\B}_p(f)\sup_{g\in B_{\widehat{\B}(\D)}}\left(\sum_{i=1}^n\left|\lambda_i\right|^p\left|g'(z_i)\right|^p\right)^{\frac{1}{p}},
\end{align*}
and thus $\lambda f\in\Pi^{\widehat{\B}}_p(\D,X)$ with $\pi^{\B}_p(\lambda f)\leq|\lambda|\pi^{\B}_p(f)$. This implies that $\pi^{\B}_p(\lambda f)=0=\left|\lambda\right|\pi^{\B}_p(f)$ if $\lambda=0$. For $\lambda\neq 0$, we have $\pi^{\B}_p(f)=\pi^{\B}_p(\lambda^{-1}(\lambda f))\leq\left|\lambda\right|^{-1}\pi^{\B}_p(\lambda f)$, hence $\left|\lambda\right|\pi^{\B}_p(f)\leq\pi^{\B}_p(\lambda f)$, and so $\pi^{\B}_p(\lambda f)=\left|\lambda\right|\pi^{\B}_p(f)$. Thus we have proved that $\left(\Pi^{\widehat{\B}}_p(\D,X),\pi^{\B}_p\right)$ is a normed space.

To show that it is a Banach space, it is enough to see that every absolutely convergent series is convergent. So let $(f_n)_{n\geq 1}$ be a sequence in $\Pi^{\widehat{\B}}_p(\D,X)$ such that $\sum\pi^{\B}_p(f_n)$ converges. Since $p_{\B}(f_n)\leq\pi^{\B}_p(f_n)$ for all $n\in\mathbb{N}$ and $\left(\widehat{\B}(\D,X),p_\B\right)$ is a Banach space, then $\sum f_n$ converges in $\left(\widehat{\B}(\D,X),p_\B\right)$ to a function $f\in\widehat{\B}(\D,X)$. Given $m\in\mathbb{N}$, $z_1,\ldots,z_m\in\D$ and $\lambda_1,\ldots,\lambda_m\in\mathbb{C}$, we have 
\begin{align*}
\left(\sum_{k=1}^m\left|\lambda_k\right|^p\left\|\sum_{i=1}^nf'_i(z_k)\right\|^p\right)^{\frac{1}{p}}
&\leq \pi^{\B}_p\left(\sum_{i=1}^nf_i\right)\sup_{g\in B_{\widehat{\B}(\D)}}\left(\sum_{k=1}^m\left|\lambda_k\right|^p\left|g'(z_k)\right|^p\right)^{\frac{1}{p}}\\
&\leq \sum_{i=1}^n\pi^{\B}_p(f_i)\sup_{g\in B_{\widehat{\B}(\D)}}\left(\sum_{k=1}^m\left|\lambda_k\right|^p\left|g'(z_k)\right|^p\right)^{\frac{1}{p}}\\
\end{align*}
for all $n\in\mathbb{N}$, and by taking limits with $n\to\infty$ yields 
$$
\left(\sum_{k=1}^m\left|\lambda_k\right|^p\left\|\sum_{i=1}^\infty f'_i(z_k)\right\|^p\right)^{\frac{1}{p}}
\leq \sum_{i=1}^\infty\pi^{\B}_p(f_i)\sup_{g\in B_{\widehat{\B}(\D)}}\left(\sum_{k=1}^m\left|\lambda_k\right|^p\left|g'(z_k)\right|^p\right)^{\frac{1}{p}}.
$$
Hence $f\in\Pi^{\widehat{\B}}_p(\D,X)$ with $\pi^{\B}_p(f)\leq\sum_{n=1}^\infty \pi^{\B}_p(f_n)$. Moreover, we have 
$$
\pi^{\B}_p\left(f-\sum_{i=1}^nf_i\right)=\pi^{\B}_p\left(\sum_{i=n+1}^\infty f_i\right)\leq\sum_{i=n+1}^\infty \pi^{\B}_p(f_i) 
$$
for all $n\in\mathbb{N}$, and thus $f$ is the $\pi^{\B}_p$-limit of the series $\sum f_n$.

(N2) Let $g\in\widehat{\B}(\D)$ and $x\in X$. Note that $g\cdot x\in\widehat{\B}(\D,X)$ with $p_\B(g\cdot x)=p_\B(g)\left\|x\right\|$ by \cite[Proposition 5.13]{JimRui-22}. If $g=0$, there is nothing to prove. Assume $g\neq 0$. We have 
\begin{align*}
\left(\sum_{i=1}^n\left|\lambda_i\right|^p\left\|(g\cdot x)'(z_i)\right\|^p\right)^{\frac{1}{p}}
&=\left\|x\right\|p_\B(g)\left(\sum_{i=1}^n\left|\lambda_i\right|^p\left|\left(\frac{g}{p_\B(g)}\right)'(z_i)\right|^p\right)^{\frac{1}{p}}\\
&\leq\left\|x\right\|p_\B(g)\sup_{h\in B_{\widehat{\B}(\D)}}\left(\sum_{i=1}^n\left|\lambda_i\right|^p\left|h'(z_i)\right|^p\right)^{\frac{1}{p}},
\end{align*}
and thus $g\cdot x\in\Pi^{\widehat{\B}}_p(\D,X)$ with $\pi^{\B}_p(g\cdot x)\leq p_\B(g)\left\|x\right\|$. Conversely, we have 
$$
p_\B(g)\left\|x\right\|=p_\B(g\cdot x)\leq\pi^{\B}_p(g\cdot x).
$$

(N3) Let $h\colon\D\to\D$ be a holomorphic function with $h(0)=0$, $f\in\Pi^{\widehat{\B}}_p(\D,X)$ and $T\in\L(X,Y)$ where $Y$ is a complex Banach space. Note that $T\circ f\circ h\in\widehat{\B}(\D,Y)$ 
by \cite[Proposition 5.13]{JimRui-22}. We have 
\begin{align*}
\left(\sum_{i=1}^n\left|\lambda_i\right|^p\left\|(T\circ f\circ h)'(z_i)\right\|^p\right)^{\frac{1}{p}}
&=\left(\sum_{i=1}^n\left|\lambda_i\right|^p\left\|T(f'(h(z_i))h'(z_i))\right\|^p\right)^{\frac{1}{p}}\\
&\leq\left\|T\right\|\left(\sum_{i=1}^n\left|\lambda_i\right|^p\left|h'(z_i)\right|^p\left\|f'(h(z_i))\right\|^p\right)^{\frac{1}{p}}\\
&\leq \left\|T\right\|\pi^{\B}_p(f)\sup_{g\in B_{\widehat{\B}(\D)}}\left(\sum_{i=1}^n\left|\lambda_i\right|^p\left|h'(z_i)\right|^p\left|g'(h(z_i))\right|^p\right)^{\frac{1}{p}}\\
&=\left\|T\right\|\pi^{\B}_p(f)\sup_{g\in B_{\widehat{\B}(\D)}}\left(\sum_{i=1}^n\left|\lambda_i\right|^p\left|(g\circ h)'(z_i)\right|^p\right)^{\frac{1}{p}}\\
&\leq\left\|T\right\|\pi^{\B}_p(f)\sup_{k\in B_{\widehat{\B}(\D)}}\left(\sum_{i=1}^n\left|\lambda_i\right|^p\left|k'(z_i)\right|^p\right)^{\frac{1}{p}},
\end{align*}
where we have used that $p_\B(g\circ h)\leq p_\B(g)$ by \cite[Proposition 3.6]{JimRui-22}. Therefore $T\circ f\circ h\in\Pi^{\widehat{\B}}_p(\D,Y)$ with $\pi^{\B}_p(T\circ f\circ h)\leq\left\|T\right\|\pi^{\B}_p(f)$.

(I) Let $f\in\widehat{\B}(\D,X)$ and let $\iota\colon X\to Y$ be a linear (not necessarily surjective) isometry. Assume that $\iota\circ f\in\Pi^{\widehat{\B}}_p(\D,Y)$. We have 
\begin{align*}
\left(\sum_{i=1}^n\left|\lambda_i\right|^p\left\|f'(z_i)\right\|^p\right)^{\frac{1}{p}}
&=\left(\sum_{i=1}^n\left|\lambda_i\right|^p\left\|\iota(f'(z_i))\right\|^p\right)^{\frac{1}{p}}=\left(\sum_{i=1}^n\left|\lambda_i\right|^p\left\|(\iota\circ f)'(z_i))\right\|^p\right)^{\frac{1}{p}}\\
&\leq \pi^{\B}_p(\iota\circ f)\sup_{g\in B_{\widehat{\B}(\D)}}\left(\sum_{i=1}^n\left|\lambda_i\right|^p\left|g'(z_i)\right|^p\right)^{\frac{1}{p}}
\end{align*}
and thus $f\in\Pi^{\widehat{\B}}_p(\D,X)$ with $\pi^{\B}_p(f)\leq\pi^{\B}_p(\iota\circ f)$. The reverse inequality follows from (N3).
\end{proof}


\subsection{M\"obius invariance}

The \textit{M\"obius group of $\D$}, denoted $\Aut(\D)$, is formed by all biholomorphic bijections $\phi\colon\D\to\D$. Each $\phi\in\Aut(\D)$ has the form $\phi=\lambda\phi_a$ with $\lambda\in\T$ and $a\in\D$, where
$$
\phi_a(z)=\frac{a-z}{1-\overline{a}z}\qquad (z\in\D).
$$

Given a complex Banach space $X$, let us recall (see \cite{AraFisPee-85}) that a linear space $\A(\D,X)$ of holomorphic mappings from $\D$ into $X$, endowed with a seminorm $p_\A$, is \textit{M\"obius-invariant} if it holds: 
\begin{enumerate}
	\item[(i)] $\A(\D,X)\subseteq \B(\D,X)$ and there exists $c>0$ such that $p_\B(f)\leq cp_\A(f)$ for all $f\in\A(\D,X)$,
	\item[(ii)] $f\circ\phi\in\A(\D,X)$ with $p_\A(f\circ\phi)=p_\A(f)$ for all $\phi\in\Aut(\D)$ and $f\in\A(\D,X)$.
\end{enumerate}

By Proposition \ref{prop-1}, each $p$-summing Bloch mapping $f\colon\D\to X$ is Bloch with $p_\B(f)\leq\pi_p^{\B}(f)$. Moreover, following the argument of the proof of (N3) in Proposition \ref{ideal summing}, it is easy to prove that if $f\colon\D\to X$ is $p$-summing Bloch and $\phi\in\Aut(\D)$, then $f\circ\phi$ is $p$-summing with $\pi^{\B}_p(f\circ\phi)\leq\pi^{\B}_p(f)$, and using this fact we also deduce that $\pi^{\B}_p(f)=\pi^{\B}_p((f\circ\phi)\circ\phi^{-1})\leq\pi^{\B}_p(f\circ\phi)$. In this way we have proved the following.

\begin{proposition}\label{prop-3-1}
$(\Pi^{\B}_p(\D,X),\pi^{\B}_p)$ is a M\"obius-invariant space for $1\leq p\leq\infty$. $\hfill\qed$ 
\end{proposition}

\subsection{Pietsch domination}

We establish a version for $p$-summing Bloch mappings on $\D$ of the known Pietsch domination Theorem for $p$-summing linear operators between Banach spaces \cite[Theorem 2]{Pie-67}. 

Let us recall that $\widehat{\B}(\D)$ is a dual Banach space (see \cite{AndCluPom-74}) and therefore we can consider this space equipped with its weak* topology. Let $\P(B_{\widehat{\B}(\D)})$ denote the set of all Borel regular probability measures $\mu$ on $(B_{\widehat{\B}(\D)},w^*)$.

\begin{theorem}\label{Pietsch} 
Let $1\leq p<\infty$ and $f\in\widehat{\B}(\D,X)$. The following statements are equivalent:
\begin{enumerate}
\item[(i)] $f$ is $p$-summing Bloch.
\item[(ii)] (Pietsch domination). There is a constant $c\geq 0$ and a Borel regular probability measure $\mu$ on $(B_{\widehat{\B}(\D)},w^*)$ such that 
$$
\left\|f'(z)\right\|\leq c\left(\int_{B_{\widehat{\B}(\D)}}\left\vert g'(z)\right\vert^{p}d\mu(g)\right)^{\frac{1}{p}}
$$
for all $z\in \D$.
\end{enumerate}
In this case, $\pi^{\B}_p(f)$ is the infimum of all constants $c\geq 0$ satisfying the preceding inequality, and this infimum is attained.
\end{theorem}

\begin{proof}
$(i) \Rightarrow (ii)$: We will apply an unified abstract version of Piestch domination Theorem (see \cite{BotPelRue-10,PelSan-11}). For it, consider the functions 
$$
S\colon\widehat{\B}(\D,X)\times\D\times\C\to [0,\infty[,\qquad S(f,z,\lambda)=\left|\lambda\right|\left\|f'(z)\right\|
$$
and 
$$
R\colon B_{\widehat{\B}(\D)}\times\D\times\C\to [0,\infty[,\qquad R(g,z,\lambda)=\left|\lambda\right|\left|g'(z)\right|.
$$
Note first that for any $z\in\D$ and $\lambda\in\C$, the function $R_{z,\lambda}\colon B_{\widehat{\B}(\D)}\to [0,\infty[$, given by 
$$
R_{z,\lambda}(g)=R(g,z,\lambda),
$$
is continuous. For every $n\in\mathbb{N}$, $\lambda_1,\ldots,\lambda_n\in\mathbb{C}$ and $z_1,\ldots,z_n\in\D$, we have  
\begin{align*}
\left(\sum_{i=1}^n S(f,z_i,\lambda_i)^p\right)^{\frac{1}{p}}&=\left(\sum_{i=1}^n \left|\lambda_i\right|^p\left\|f'(z_i)\right\|^p\right)^{\frac{1}{p}}\\
&\leq\pi^{\B}_p(f)\sup_{g\in B_{\widehat{\B}(\D)}}\left(\sum_{i=1}^n\left|\lambda_i\right|^p\left|g'(z_i)\right|^p\right)^{\frac{1}{p}}\\
&=\pi^{\B}_p(f)\sup_{g\in B_{\widehat{\B}(\D)}}\left(\sum_{i=1}^n R(g,z_i,\lambda_i)^p\right)^{\frac{1}{p}},
\end{align*}
and therefore $f$ is $R-S$-abstract $p$-summing. Hence, by applying \cite[Theorem 3.1]{PelSan-11}, there is a constant $c\geq 0$ and a measure $\mu\in\P(B_{\widehat{\B}(\D)})$ such that 
$$
S(f,z,\lambda)\leq c\left(\int_{B_{\widehat{\B}(\D)}}R(g,z,\lambda)^p\ d\mu(g)\right)^{\frac{1}{p}}
$$
for all $z\in \D$ and $\lambda\in\C$, and therefore 
$$
\left\|f'(z)\right\|\leq c\left(\int_{B_{\widehat{\B}(\D)}}\left\vert g'(z)\right\vert^{p}d\mu(g)\right)^{\frac{1}{p}}
$$
for all $z\in \D$. Furthermore, we have  
$$
\left\|f'(z)\right\|=\left(\sum_{i=1}^n\left|\lambda_i\right|^p\left\|f'(z_i)\right\|^p\right)^{\frac{1}{p}}\leq \pi^{\B}_p(f)\left(\int_{B_{\widehat{\B}(\D)}}\left\vert g'(z)\right\vert^{p}d\mu(g)\right)^{\frac{1}{p}}
$$
for every $z\in\D$ by taking, for example, $n\in\mathbb{N}$, $\lambda_1=1$, $\lambda_2=\cdots=\lambda_n=0$ and $z_1=\cdots=z_n=z$.

$(ii)\Rightarrow(i)$: Given $n\in\mathbb{N}$, $\lambda_1,\ldots,\lambda_n\in\mathbb{C}$ and $z_1,\ldots,z_n\in\D$, we have
\begin{align*}
\left(\sum_{i=1}^n\left|\lambda_i\right|^p\left\|f'(z_i)\right\|^p\right)^{\frac{1}{p}}
&\leq c \sum_{i=1}^n\left(\int_{B_{\widehat{\B}(\D)}}\left|\lambda_i\right|^p\left\vert g'(z_i)\right\vert^{p}d\mu(g)\right)^{\frac{1}{p}}\\
&\leq c \sup_{g\in B_{\widehat{\B}(\D)}}\left(\sum_{i=1}^n\left|\lambda_i\right|^p\left|g'(z_i)\right|^p\right)^{\frac{1}{p}}.
\end{align*}
Hence $f\in\Pi^{\widehat{\B}}_p(\D,X)$ with $\pi^{\B}_p(f)\leq c$. 
\end{proof}

\subsection{Pietsch factorization}

We now present the analogue for $p$-summing Bloch mappings of Pietsch factorization theorem for $p$-summing operators (see \cite[Theorem 3]{Pie-67}, also \cite[Theorem 2.13]{DisJarTon-95}). 

Given $\mu\in\P(B_{\widehat{\B}(\D)})$ and $1\leq p<\infty$, $I_{\infty,p}\colon L_\infty(\mu)\to L_p(\mu)$ and $j_{\infty}\colon C(B_{\widehat{\B}(\D)})\to L_\infty(\mu)$ denote the formal inclusion operators. We will also use the mapping $\iota_\D\colon\D\to C(B_{\widehat{\B}(\D)})$ defined by  
$$
\iota_\D(z)(g)=g'(z)\quad \left(z\in\D,\; g\in B_{\widehat{\B}(\D)}\right),
$$
and for a complex Banach space $X$, the isometric linear embedding $\iota_X\colon X\to \ell_\infty(B_{X^*})$ given by 
$$
\left\langle \iota_X(x),x^*\right\rangle=x^*(x)\quad (x^*\in B_{X^*},\; x\in X).
$$

The following easy fact will be applied below.

\begin{lemma}\label{main lemma}
Let $\mu\in\P(B_{\widehat{\B}(\D)})$. Then there exists a mapping $h\in\widehat{\B}(\D,L_\infty(\mu))$ with $p_\B(h)=1$ such that $h'=j_{\infty}\circ\iota_\D$. In fact, $h\in\Pi^{\widehat{\B}}_p(\D,L_\infty(\mu))$ with $\pi_p^\B(h)=1$ for any $1\leq p<\infty$.
\end{lemma} 

\begin{proof}
Note that $j_{\infty}\circ\iota_\D\in\H(\D,L_\infty(\mu))$ with $(j_{\infty}\circ\iota_\D)'=j_{\infty}\circ (\iota_\D)'$, where $(\iota_\D)'(z)(g)=g''(z)$ for all $z\in\D$ and $g\in B_{\widehat{\B}(\D)}$. By \cite[Lemma 2.9]{JimRui-22}, there exists a mapping $h\in\H(\D,L_\infty(\mu))$ with $h(0)=0$ such that $h'=j_{\infty}\circ\iota_\D$. In fact, $h\in\widehat{\B}(\D,L_\infty(\mu))$ with $p_\B(h)=1$ since 
$$
(1-|z|^2)\left\|h'(z)\right\|_{L_\infty(\mu)}=(1-|z|^2)\left\|j_{\infty}(\iota_\D(z))\right\|_{L_\infty(\mu)}=(1-|z|^2)\left\|\iota_\D(z)\right\|_{\infty}=1
$$
for all $z\in\D$. For the second assertion, given $1\leq p<\infty$, it suffices to note that 
\begin{align*}
\left(\sum_{i=1}^n\left|\lambda_i\right|^p\left\|h'(z_i)\right\|_{L_\infty(\mu)}^p\right)^{\frac{1}{p}}
&=\left(\sum_{i=1}^n\left|\lambda_i\right|^p\left\|j_{\infty}(\iota_\D(z_i))\right\|_{L_\infty(\mu)}^p\right)^{\frac{1}{p}}
=\left(\sum_{i=1}^n\left|\lambda_i\right|^p\left\|\iota_\D(z_i)\right\|_{\infty}^p\right)^{\frac{1}{p}}\\
&=\left(\sum_{i=1}^n\frac{\left|\lambda_i\right|^p}{(1-|z_i|^2)^p}\right)^{\frac{1}{p}}
=\left(\sum_{i=1}^n\left|\lambda_i\right|^p\left|f'_{z_i}(z_i)\right|^p\right)^{\frac{1}{p}}\\
&\leq\sup_{g\in B_{\widehat{\B}(\D)}}\left(\sum_{i=1}^n\left|\lambda_i\right|^p\left|g'(z_i)\right|^p\right)^{\frac{1}{p}}
\end{align*}
for any $n\in\mathbb{N}$, $\lambda_1,\ldots,\lambda_n\in\mathbb{C}$ and $z_1,\ldots,z_n\in\D$.
\end{proof}

\begin{theorem}\label{Pietsch2} 
Let $1\leq p<\infty$ and $f\in\widehat{\B}(\D,X)$. The following assertions are equivalent:
\begin{enumerate}
\item[(i)] $f$ is $p$-summing Bloch.
\item[(ii)] (Pietsch factorization). There exist a regular Borel probability measure $\mu$ on $(B_{\widehat{\B}(\D)},w^*)$, an operator $T\in\L(L_p(\mu),\ell_\infty(B_{X^*}))$ and a mapping $h\in\widehat{\B}(\D,L_\infty(\mu))$ such that the following diagram commutes:
$$			
\xymatrix{L_\infty(\mu) \ar[rr]^{I_{\infty,p}} && L_p(\mu)\ar[d]^{T} \\
				\D\ar[u]^{h'} \ar[r]^{f'} & X \ar[r]^{\iota_X} & \ell_\infty(B_{X^*})}
$$
\end{enumerate}

In this case, $\pi^{\B}_p(f)=\inf\left\{\left\|T\right\|p_\B(h)\right\}$, where the infimum is taken over all such factorizations of $\iota_X\circ f'$ as above, and this infimum is attained.
\end{theorem}

\begin{proof}
$(i) \Rightarrow (ii)$: If $f\in\Pi^{\widehat{\B}}_p(\D,X)$, then Theorem \ref{Pietsch} gives a measure $\mu\in\P(B_{\widehat{\B}(\D)})$ such that 
$$
\left\|f'(z)\right\|\leq \pi^{\B}_p(f)\left(\int_{B_{\widehat{\B}(\D)}}\left\vert g'(z)\right\vert^{p}d\mu(g)\right)^{\frac{1}{p}}
$$
for all $z\in \D$. By Lemma \ref{main lemma}, there is a mapping $h\in\widehat{\B}(\D,L_\infty(\mu))$ with $p_\B(h)=1$ such that $h'=j_{\infty}\circ\iota_\D$. Consider the linear subspace $S_p:=\overline{\lin}(I_{\infty,p}(h'(\D)))\subseteq L_p(\mu)$ and the operator $T_0\in\L(S_p,\ell_\infty(B_{X^*}))$ defined by $T_0(I_{\infty,p}(h'(z)))=\iota_X(f'(z))$ for all $z\in\D$. Note that $\left\|T_0\right\|\leq\pi^{\B}_p(f)$ since 
\begin{align*}
\left\|T_0\left(\sum_{i=1}^n\alpha_iI_{\infty,p}(h'(z_i))\right)\right\|_\infty&=\left\|\sum_{i=1}^n\alpha_i T_0(I_{\infty,p}(h'(z_i)))\right\|_\infty
=\left\|\sum_{i=1}^n\alpha_i \iota_X(f'(z_i))\right\|_\infty\\
&\leq\sum_{i=1}^n\left|\alpha_i\right|\left\|\iota_X(f'(z_i))\right\|_\infty
=\sum_{i=1}^n\left|\alpha_i\right|\left\|f'(z_i)\right\|\\
&\leq\pi^{\B}_p(f)\sum_{i=1}^n\left|\alpha_i\right|\left(\int_{B_{\widehat{\B}(\D)}}\left|g'(z_i)\right|^{p}d\mu(g)\right)^{\frac{1}{p}}\\
&\leq\pi^{\B}_p(f)\sum_{i=1}^n\frac{\left|\alpha_i\right|}{1-|z_i|^2}
\end{align*}
and 
\begin{align*}
\sum_{i=1}^n\frac{\left|\alpha_i\right|}{1-|z_i|^2}
&=\left|\sum_{i=1}^n\alpha_i\frac{\overline{\alpha_i}}{\left|\alpha_i\right|}f'_{z_i}(z_i)\right|
=\sup_{g\in B_{\widehat{\B}(\D)}}\left|\sum_{i=1}^n\alpha_i g'(z_i)\right|=\sup_{g\in B_{\widehat{\B}(\D)}}\left|\sum_{i=1}^n\alpha_i \iota_\D(z_i)(g)\right|\\
&=\left\|\sum_{i=1}^n\alpha_i\iota_\D(z_i)\right\|_\infty=\left\|\sum_{i=1}^n\alpha_i j_{\infty}(\iota_\D(z_i))\right\|_{L_\infty(\mu)}=\left\|\sum_{i=1}^n\alpha_i h'(z_i)\right\|_{L_\infty(\mu)}\\
&=\left\|I_{\infty,p}\left(\sum_{i=1}^n\alpha_ih'(z_i)\right)\right\|_{L_p(\mu)}=\left\|\sum_{i=1}^n\alpha_iI_{\infty,p}(h'(z_i))\right\|_{L_p(\mu)}
\end{align*}
for any $n\in\mathbb{N}$, $\alpha_1,\ldots,\alpha_n\in\mathbb{C}^*$ and $z_1,\ldots,z_n\in\D$. By the injectivity of the Banach space $\ell_\infty(B_{X^*})$ (see \cite[p. 45]{DisJarTon-95}), there exists $T\in\L(L_p(\mu),\ell_\infty(B_{X^*}))$ such that $\left.T\right|_{S_p}=T_0$ with $\left\|T\right\|=\left\|T_0\right\|$. This tells us that $\iota_X\circ f'=T\circ I_{\infty,p}\circ h'$ with $\left\|T\right\|p_\B(h)\leq\pi^{\B}_p(f)$.  

$(ii)\Rightarrow(i)$: By (ii), we have $\iota_X\circ f'=T\circ I_{\infty,p}\circ h'$. Given $n\in\mathbb{N}$, $\lambda_1,\ldots,\lambda_n\in\mathbb{C}$ and $z_1,\ldots,z_n\in\D$, it holds 
\begin{align*}
\left(\sum_{i=1}^n\left|\lambda_i\right|^p\left\|f'(z_i)\right\|^p\right)^{\frac{1}{p}}&=\left(\sum_{i=1}^n\left|\lambda_i\right|^p\left\|\iota_X(f'(z_i))\right\|_\infty^p\right)^{\frac{1}{p}}
=\left(\sum_{i=1}^n\left|\lambda_i\right|^p\left\|T(I_{\infty,p}(h'(z_i)))\right\|_\infty^p\right)^{\frac{1}{p}}\\
&\leq \left\|T\right\|\left(\sum_{i=1}^n\left|\lambda_i\right|^p\left\|I_{\infty,p}(h'(z_i))\right\|_{L_p(\mu)}^p\right)^{\frac{1}{p}}\\
&=\left\|T\right\|\left(\sum_{i=1}^n\left|\lambda_i\right|^p\left\|h'(z_i)\right\|_{L_\infty(\mu)}^p\right)^{\frac{1}{p}}\\
&\leq\left\|T\right\|p_\B(h)\left(\sum_{i=1}^n\frac{\left|\lambda_i\right|^p}{(1-|z_i|^2)^p}\right)^{\frac{1}{p}}\\
&=\left\|T\right\|p_\B(h)\left(\sum_{i=1}^n\left|\lambda_i\right|^p\left|f'_{z_i}(z_i)\right|^p\right)^{\frac{1}{p}}\\
&\leq\left\|T\right\|p_\B(h)\sup_{g\in B_{\widehat{\B}(\D)}}\left(\sum_{i=1}^n\left|\lambda_i\right|^p\left|g'(z_i)\right|^p\right)^{\frac{1}{p}}.
\end{align*}
Hence $f\in\Pi^{\widehat{\B}}_p(\D,X)$ with $\pi^{\B}_p(f)\leq \left\|T\right\|p_\B(h)$.
\end{proof}

The concept of holomorphic mapping with a relatively (weakly) compact Bloch range was introduced in \cite{JimRui-22}. The \textit{Bloch range} of a function $f\in\H(\D,X)$ is the set 
$$
\rang_{\B}(f):=\left\{(1-|z|^2)f'(z)\colon z\in\D\right\}\subseteq X.
$$
A mapping $f\in\H(\D,X)$ is said to be \textit{(weakly) compact Bloch} if $\rang_{\B}(f)$ is a relatively (weakly) compact subset of $X$.


\begin{corollary}
Let $1\leq p<\infty$.
\begin{enumerate}
\item[(i)] Every $p$-summing Bloch mapping from $\D$ to $X$ is weakly compact Bloch. 
\item[(ii)] If $X$ is reflexive, then every $p$-summing Bloch mapping from $\D$ to $X$ is compact Bloch.
\end{enumerate}
\end{corollary}

\begin{proof}
(i) Assume first $p>1$. If $f\in\Pi^{\widehat{\B}}_p(\D,X)$, then Theorem \ref{Pietsch2} gives a regular Borel probability measure $\mu$ on $(B_{\widehat{\B}(\D)},w^*)$, an operator $T\in\L(L_p(\mu),\ell_\infty(B_{X^*}))$ and a map $h\in\widehat{\B}(\D,L_\infty(\mu))$ such that $\iota_X\circ f'=T\circ I_{\infty,p}\circ h'$, that is, $(\iota_X\circ f)'=T\circ (I_{\infty,p}\circ h)'$. Since $L_p(\mu)$ is reflexive, it follows that $\iota_X\circ f\in\widehat{\B}(\D,\ell_\infty(B_{X^*}))$ is weakly compact Bloch by \cite[Theorem 5.6]{JimRui-22}. Since $\rang_{\B}(\iota_X\circ f)=\iota_X(\rang_{\B}(f))$, we conclude that $f$ is weakly compact Bloch. For $p=1$, the result follows from Proposition \ref{prop-1} and from what was proved above. 

(ii) It follows from (i) that if $f\in\Pi^{\widehat{\B}}_p(\D,X)$ and $X$ is reflexive, then $\rang_{\B}(f)$ is relatively compact in $X$, hence $f$ is compact Bloch.
\end{proof}

\subsection{Maurey extrapolation}

We now use Pietsch domination of $p$-summing Bloch mappings to give a Bloch version of Maurey's extrapolation Theorem \cite{Mau-74}.

\begin{theorem}\label{Pietsch3} 
Let $1<p<q<\infty$ and assume that $\Pi^{\widehat{\B}}_q(\D,\ell_q)=\Pi^{\widehat{\B}}_p(\D,\ell_q)$. Then $\Pi^{\widehat{\B}}_q(\D,X)=\Pi^{\widehat{\B}}_1(\D,X)$ for every complex Banach space $X$. 
\end{theorem}

\begin{proof}
Lemma \ref{main lemma} and Proposition \ref{ideal summing} assures that for each $\mu\in\P(B_{\widehat{\B}(\D)})$, there is a mapping $h_\mu\in\widehat{\B}(\D,L_\infty(\mu))$ such that $h_\mu'=j_{\infty}\circ\iota_\D$ and $I_{\infty,q}\circ h_\mu\in\Pi^{\widehat{\B}}_q(\D,L_q(\mu))$ with $\pi_q^\B(I_{\infty,q}\circ h_\mu)\leq 1$.

We now follow the proof of \cite[Theorem 3.17]{DisJarTon-95}. Since $\Pi^{\widehat{\B}}_q(\D,\ell_q)=\Pi^{\widehat{\B}}_p(\D,\ell_q)$ and $\pi^{\B}_q\leq\pi^{\B}_p$ on $\Pi^{\widehat{\B}}_p(\D,\ell_q)$ by Proposition \ref{prop-1}, the Closed Graph Theorem yields a constant $c>0$ such that $\pi_{p}^{\B}(f)\leq c\pi_{q}^{\B}(f)$ for all $f\in\Pi^{\widehat{\B}}_q(\D,\ell_q)$. Since $L_q(\mu)$ is an $\mathcal{L}_{q,\lambda}$-space for each $\lambda>1$, we can assure that given $n\in\N$ and $z_1,\ldots,z_n\in\D$, the subspace 
$$
E=\lin\left(\left\{I_{\infty,q}(h_\mu(z_1)),\ldots,I_{\infty,q}(h_\mu(z_n))\right\}\right)\subseteq L_q(\mu)
$$ 
embeds $\lambda$-isomorphically into $\ell_q$, that is, $E$ is contained in a subspace $F\subseteq L_q(\mu)$ for which there exists an isomorphism $T\colon F\to\ell_{q}$ with $\left\|T\right\|||T^{-1}||<\lambda$. 

Since $T\circ I_{\infty,q}\circ h_\mu\in\Pi^{\widehat{\B}}_q(\D,\ell_q)=\Pi^{\widehat{\B}}_p(\D,\ell_q)$ and $(T\circ I_{\infty,q}\circ h_\mu)'=T\circ I_{\infty,q}\circ h_\mu'$, we have 
\begin{align*}
&\left(\sum_{i=1}^n\left|\lambda_i\right|^p\left\|(I_{\infty,q}\circ h_\mu)'(z_i)\right\|_{L_q(\mu)}^p\right)^{\frac{1}{p}}
\leq \left\|T^{-1}\right\|\left(\sum_{i=1}^n\left|\lambda_i\right|^p\left\|T(I_{\infty,q}(h_\mu'(z_i)))\right\|_{\ell_{q}}^p\right)^{\frac{1}{p}}\\
&\leq \left\|T^{-1}\right\|c\pi_q^{\B}(T\circ I_{\infty,q}\circ h_\mu)\sup_{g\in B_{\widehat{\B}(\D)}}\left(\sum_{i=1}^n\left|\lambda_i\right|^p\left|g'(z_i)\right|^p\right)^{\frac{1}{p}}\\
&\leq c\left\|T^{-1}\right\|\left\|T\right\|\pi_q^{\B}(I_{\infty,q}\circ h_\mu)\sup_{g\in B_{\widehat{\B}(\D)}}\left(\sum_{i=1}^n\left|\lambda_i\right|^p\left|g'(z_i)\right|^p\right)^{\frac{1}{p}},
\end{align*}
therefore $\pi_p^{\B}(I_{\infty,q}\circ h_\mu)\leq c\lambda$ for all $\lambda>1$, and thus $\pi_p^{\B}(I_{\infty,q}\circ h_\mu)\leq c$. Now, by Theorem \ref{Pietsch}, there exists a measure $\widehat{\mu}\in\P(B_{\widehat{\B}(\D)})$ such that
$$
\left\|(I_{\infty,q}\circ h_\mu)'(z)\right\|_{L_q(\mu)}
\leq c\left(\int_{B_{\widehat{\B}(\D)}}\left|g'(z)\right|^p\ d\widehat{\mu}(g)\right)^{\frac{1}{p}}
=c\left\|(I_{\infty,q}\circ h_{\widehat{\mu}})'(z)\right\|_{L_p(\widehat{\mu})}
$$
for all $z\in\D$. In the last equality, we have used that 
$$
(I_{\infty,q}\circ h_{\widehat{\mu}})'(z)(g)=I_{\infty,q}(h'_{\widehat{\mu}}(z))(g)=h'_{\widehat{\mu}}(z)(g)=
j_{\infty}(\iota_\D(z))(g)=\iota_\D(z)(g)=g'(z)
$$
for all $z\in\D$ and $g\in B_{\widehat{\B}(\D)}$. 

Take a complex Banach space $X$ and let $f\in\Pi^{\widehat{\B}}_q(\D,X)$. In view of Proposition \ref{prop-1}, we only must show that $f\in\Pi^{\widehat{\B}}_1(\D,X)$. Theorem \ref{Pietsch} provides again a measure $\mu_0\in\P(B_{\widehat{\B}(\D)})$ such that
$$
\left\|f'(z)\right\|\leq\pi_q^{\B}(f)\left\|(I_{\infty,q}\circ h_{\mu_0})'(z)\right\|_{L_q(\mu_0)}
$$
for all $z\in\D$. We claim that there is a constant $C>0$ and a measure $\lambda\in\P(B_{\widehat{\B}(\D)})$ such that
$$
\left\|(I_{\infty,q}\circ h_{\mu_0})'(z)\right\|_{L_q(\mu_0)}\leq C\left\|(I_{\infty,q}\circ h_{\lambda})'(z)\right\|_{L_1(\lambda)}
$$
for all $z\in\D$. Indeed, define $\lambda=\sum_{n=0}^\infty(1/2^{n+1})\mu_n\in\P(B_{\widehat{\B}(\D)})$, where $(\mu_n)_{n\geq 1}$ is the sequence in $\P(B_{\widehat{\B}(\D)})$ given by $\mu_{n+1}=\widehat{\mu_n}$ for all $n\in\N_0$, where the measure $\widehat{\mu_n}$ is defined using Theorem \ref{Pietsch}. Since $1<p<q$, there exists $\theta\in (0,1)$ such that $p=\theta\cdot 1+(1-\theta)q$, and applying H\"older's Inequality with $1/\theta$ (note that $(1/\theta)^*=1/(1-\theta)$), we have 
\begin{align*}
&\left\|(I_{\infty,q}\circ h_{\mu_n})'(z)\right\|_{L_p(\mu_n)}
=\left(\int_{B_{\widehat{\B}(\D)}}\left|(I_{\infty,q}\circ h_{\mu_n})'(z)(g)\right|^{\theta\cdot 1+(1-\theta)q}\ d\mu_n(g)\right)^{\frac{1}{p}}\\
&\leq\left(\int_{B_{\widehat{\B}(\D)}}\left|(I_{\infty,q}\circ h_{\mu_n})'(z)(g)\right|\ d\mu_n(g)\right)^{\theta}
\left(\int_{B_{\widehat{\B}(\D)}}\left|(I_{\infty,q}\circ h_{\mu_n})'(z)(g)\right|^q\ d\mu_n(g)\right)^{\frac{1-\theta}{q}}\\
&=\left\|(I_{\infty,q}\circ h_{\mu_n})'(z)\right\|^{\theta}_{L_1(\mu_n)}\left\|(I_{\infty,q}\circ h_{\mu_n})'(z)\right\|^{1-\theta}_{L_q(\mu_n)}
\end{align*}
for each $n\in\N_0$ and all $z\in\D$. Using H\"older's Inequality and the inequality 
\begin{align*}
\sum_{n=0}^\infty&\frac{1}{2^{n+1}}\left\|(I_{\infty,q}\circ h_{\mu_{n+1}})'(z)\right\|_{L_q(\mu_{n+1})}\\
&\leq\sum_{n=0}^\infty\frac{1}{2^{n+1}}\left\|(I_{\infty, q}\circ h_{\mu_{n+1}})'(z)\right\|_{L_q(\mu_{n+1})}+\left\|(I_{\infty,q}\circ h_{\mu_0})'(z)\right\|_{L_q(\mu_0)}\\ 
&=\sum_{n=-1}^\infty\frac{1}{2^{n+1}}\left\|(I_{\infty, q}\circ h_{\mu_{n+1}})'(z)\right\|_{L_q(\mu_{n+1})}\\ 
&=\sum_{n=0}^\infty\frac{1}{2^{n}}\left\|(I_{\infty, q}\circ h_{\mu_{n}})'(z)\right\|_{L_q(\mu_{n})}\\ 
&=2\sum_{n=0}^\infty\frac{1}{2^{n+1}}\left\|(I_{\infty,q}\circ h_{\mu_{n}})'(z)\right\|_{L_q(\mu_{n})},
\end{align*}
we now obtain 
\begin{align*}
\sum_{n=0}^\infty&\frac{1}{2^{n+1}}\left\|(I_{\infty,q}\circ h_{\mu_n})'(z)\right\|_{L_q(\mu_n)}\\
&\leq c\sum_{n=0}^\infty\frac{1}{2^{n+1}}\left\|(I_{\infty,q}\circ h_{\mu_{n+1}})'(z)\right\|_{L_p(\mu_{n+1})}\\
&\leq c\sum_{n=0}^\infty\frac{1}{2^{n+1}}\left\|(I_{\infty,q}\circ h_{\mu_{n+1}})'(z)\right\|^{\theta}_{L_1(\mu_{n+1})}\left\|(I_{\infty,q}\circ h_{\mu_{n+1}})'(z)\right\|^{1-\theta}_{L_q(\mu_{n+1})}\\
&\leq c\left(\sum_{n=0}^\infty\frac{1}{2^{n+1}}\left\|(I_{\infty,q}\circ h_{\mu_{n+1}})'(z)\right\|_{L_1(\mu_{n+1})}\right)^{\theta}
\left(\sum_{n=0}^\infty\frac{1}{2^{n+1}}\left\|(I_{\infty,q}\circ h_{\mu_{n+1}})'(z)\right\|_{L_q(\mu_{n+1})}\right)^{1-\theta}\\
&\leq c\left(\sum_{n=0}^\infty\frac{1}{2^{n+1}}\left\|(I_{\infty,q}\circ h_{\mu_{n+1}})'(z)\right\|_{L_1(\mu_{n+1})}\right)^{\theta}
\left(2\sum_{n=0}^\infty\frac{1}{2^{n+1}}\left\|(I_{\infty,q}\circ h_{\mu_{n}})'(z)\right\|_{L_q(\mu_{n})}\right)^{1-\theta}
\end{align*}
for all $z\in\D$, and thus
\begin{align*}
\sum_{n=0}^\infty\frac{1}{2^{n+1}}\left\|(I_{\infty,q}\circ h_{\mu_n})'(z)\right\|_{L_q(\mu_n)}
&\leq c^{\frac{1}{\theta}}2^{\frac{1-\theta}{\theta}}\left(\sum_{n=0}^\infty\frac{1}{2^{n+1}}\left\|(I_{\infty,q}\circ h_{\mu_{n+1}})'(z)\right\|_{L_1(\mu_{n+1})}\right)\\
&\leq c^{\frac{1}{\theta}}2^{\frac{1-\theta}{\theta}}2\left(\sum_{n=0}^\infty\frac{1}{2^{n+1}}\left\|(I_{\infty,q}\circ h_{\mu_{n}})'(z)\right\|_{L_1(\mu_{n})}\right)\\
&=(2c)^{\frac{1}{\theta}}\left\|(I_{\infty,q}\circ h_{\lambda})'(z)\right\|_{L_1(\lambda)}
\end{align*}
for all $z\in\D$. From above, we deduce that 
$$
\frac{1}{2}\left\|(I_{\infty,q}\circ h_{\mu_0})'(z)\right\|_{L_q(\mu_0)}\leq (2c)^{\frac{1}{\theta}}\left\|(I_{\infty,q}\circ h_{\lambda})'(z)\right\|_{L_1(\lambda)}
$$
for all $z\in\D$, and this proves our claim taking $C=2(2c)^{\frac{1}{\theta}}$. Therefore we can write
$$
\left\|f'(z)\right\|
\leq C\pi_q^{\B}(f)\left\|(I_{\infty,q}\circ h_{\lambda})'(z)\right\|_{L_1(\lambda)}
=C\pi_q^{\B}(f)\int_{B_{\widehat{\B}(\D)}}\left|g'(z)\right|\ d\lambda(g)
$$
for all $z\in\D$. Hence $f\in\Pi^{\widehat{\B}}_1(\D,X)$ with $\pi_1^{\B}(f)\leq C\pi_q^{\B}(f)$ by Theorem \ref{Pietsch}.
\end{proof}


\section{Banach-valued Bloch molecules on the unit disc}\label{3}

Our aim in this section is to study the duality of the spaces of $p$-summing Bloch mappings from $\D$ into $X^*$. We begin by recalling some concepts and results stated in \cite{JimRui-22} on the Bloch-free Banach space over $\D$. 

For each $z\in\D$, a \textit{Bloch atom of} $\D$ is the bounded linear functional $\gamma_z\colon\widehat{\B}(\D)\to\mathbb{C}$ given by 
$$
\gamma_z(f)=f'(z)\qquad (f\in\widehat{\B}(\D)).
$$
The elements of $\lin(\{\gamma_z\colon z\in\D\})$ in $\widehat{\B}(\D)^*$ are called \textit{Bloch molecules of} $\D$. The \textit{Bloch-free Banach space over} $\D$, denoted $\G(\D)$, is the norm-closed linear hull of $\left\{\gamma_z\colon z\in\D\right\}$ in $\widehat{\B}(\D)^*$. The mapping $\Gamma\colon\D\to\G(\D)$, defined by $\Gamma(z)=\gamma_z$ for all $z\in\D$, is holomorphic with $\left\|\gamma_z\right\|=1/(1-|z|^2)$ for all $z\in\D$ (see \cite[Proposition 2.7]{JimRui-22}).

Let $X$ be a complex Banach space. Given $z\in\D$ and $x\in X$, it is immediate that the functional $\gamma_z\otimes x\colon\widehat{\B}(\D,X^*)\to\mathbb{C}$ defined by  
$$
\left(\gamma_z\otimes x\right)(f)=\left\langle f'(z),x\right\rangle\qquad \left(f\in\widehat{\B}(\D,X^*)\right), 
$$
is linear and continuous with $\left\|\gamma_z\otimes x\right\|\leq \left\|x\right\|/(1-|z|^2)$. In fact, it is immediate that $\left\|\gamma_z\otimes x\right\|=\left\|x\right\|/(1-|z|^2)$. Indeed, take any $x^*\in S_{X^*}$ such that $x^*(x)=||x||$ and consider $f_z\cdot x^*\in\widehat{\B}(\D,X^*)$. Since $p_\B(f_z\cdot x^*)=1$, it follows that 
\begin{align*}
\left\|\gamma_z\otimes x\right\|&\geq\left|(\gamma_z\otimes x)(f_z\cdot x^*)\right|
=\left|\left\langle (f_z\cdot x^*)'(z),x\right\rangle\right|\\
&=\left|\left\langle f_z'(z)x^*,x\right\rangle\right|=\left|f_z'(z)\right|\left|x^*(x)\right|=\frac{\left\|x\right\|}{1-\left|z\right|^2}.
\end{align*}

We now present a tensor product space whose elements, according to \cite[Definition 2.6]{JimRui-22}, could be referred to as \textit{$X$-valued Bloch molecules on $\D$}.

\begin{definition}\label{def-Hol-tensor-product} 
Let $X$ be a complex Banach space. Define the linear space 
$$
\lin(\Gamma(\D))\otimes X:=\lin\left\{\gamma_z\otimes x\colon z\in\D,\, x\in X\right\}\subseteq\widehat{\B}(\D,X^*)^*.
$$
\end{definition}

Note that each element $\gamma\in\lin(\Gamma(\D))\otimes X$ is of the form 
$$
\gamma=\sum_{i=1}^n\lambda_i(\gamma_{z_i}\otimes x_i)=\sum_{i=1}^n\lambda_i\gamma_{z_i}\otimes x_i=\sum_{i=1}^n\gamma_{z_i}\otimes \lambda_i x_i
$$
where $n\in\mathbb{N}$, $\lambda_i\in\mathbb{C}$, $z_i\in \D$ and $x_i\in X$ for $i=1,\ldots,n$, but such a representation of $\gamma$ is not unique.

The action of the functional $\gamma=\sum_{i=1}^n \lambda_i\gamma_{z_i}\otimes x_i\in\lin(\Gamma(\D))\otimes X$ on a mapping $f\in\widehat{\B}(\D,X^*)$ can be described as  
$$
\gamma(f)=\sum_{i=1}^n\lambda_i\left\langle f'(z_i),x_i\right\rangle . 
$$


\subsection{Pairing}

The space $\lin(\Gamma(\D))\otimes X$ is a linear subspace of $\widehat{\B}(\D,X^*)^*$ and, in fact, we have:

\begin{proposition}\label{theo-dual-pair} 
$\left\langle\lin(\Gamma(\D))\otimes X,\widehat{\B}(\D,X^*)\right\rangle$ is a dual pair, via the bilinear form given by  
$$
\left\langle\gamma,f\right\rangle=\sum_{i=1}^n \lambda_i\left\langle f'(z_i),x_i\right\rangle 
$$
for $\gamma=\sum_{i=1}^n\lambda_i \gamma_{z_i}\otimes x_i\in\lin(\Gamma(\D))\otimes X$ and $f\in\widehat{\B}(\D,X^*)$.
\end{proposition}

\begin{proof}
Note that $\langle\cdot,\cdot\rangle$ is a well-defined bilinear map on $(\lin(\Gamma(\D))\otimes X)\times\widehat{\B}(\D,X^*)$ since $\langle\gamma,f\rangle=\gamma(f)$. On one hand, if $\gamma\in\lin(\Gamma(\D))\otimes X$ and $\langle\gamma,f\rangle=0$ for all $f\in\widehat{\B}(\D,X^*)$, then $\gamma=0$, and thus $\widehat{\B}(\D,X^*)$ separates points of $\lin(\Gamma(\D))\otimes X$. On the other hand, if $f\in\widehat{\B}(\D,X^*)$ and $\langle\gamma,f\rangle=0$ for all $\gamma\in\lin(\Gamma(\D))\otimes X$, then $\left\langle f'(z),x\right\rangle=\left\langle\gamma_z\otimes x,f\right\rangle=0$ for all $z\in \D$ and $x\in X$, hence $f'(z)=0$ for all $z\in\D$, therefore $f$ is a constant function on $\D$, then $f=0$ since $f(0)=0$ and thus $\lin(\Gamma(\D))\otimes X$ separates points of $\widehat{\B}(\D,X^*)$. 
\end{proof}

Since $\left\langle\lin(\Gamma(\D))\otimes X,\widehat{\B}(\D,X^*)\right\rangle$ is a dual pair, we can identify $\widehat{\B}(\D,X^*)$ with a linear subspace of $(\lin(\Gamma(\D))\otimes X)'$ (the \emph{algebraic dual of} $\lin(\Gamma(\D))\otimes X$) by means of the following easy result.

\begin{corollary}\label{linearization} 
For each $f\in\widehat{\B}(\D,X^*)$, the functional $\Lambda_0(f)\colon\lin(\Gamma(\D))\otimes X\to\mathbb{C}$, given by 
$$
\Lambda_0(f)(\gamma)=\sum_{i=1}^n\lambda_i\left\langle f'(z_i),x_i\right\rangle 
$$
for $\gamma=\sum_{i=1}^n \lambda_i\gamma_{z_i}\otimes x_i\in\lin(\Gamma(\D))\otimes X$, is linear. We will say that $\Lambda_0(f)$ is the linear functional on $\lin(\Gamma(\D))\otimes X$ associated to $f$. Furthermore, the map $f\mapsto \Lambda_0(f)$ is a linear monomorphism from $\widehat{\B}(\D,X^*)$ into $(\lin(\Gamma(\D))\otimes X)'$.$\hfill$ $\Box$
\end{corollary}



\subsection{Projective norm}
 
As usual (see \cite{Rya-02}), given two linear spaces $E$ and $F$, the tensor product space $E\otimes F$ equipped with a norm $\alpha$ will be denoted by $E\otimes_\alpha F$, and the completion of $E\otimes_\alpha F$ by $E\widehat{\otimes}_\alpha F$. An important example of tensor norm is the projective norm $\pi$ on $u\in E\otimes F$ defined by 
$$
\pi(u)=\inf\left\{\sum_{i=1}^n\left\|x_i\right\|\left\|y_i\right\|\colon n\in\N,\, x_1,\ldots,x_n\in E,\, y_1,\ldots,y_n\in F,\, u=\sum_{i=1}^n x_i\otimes y_i\right\},
$$
where the infimum is taken over all the representations of $u$ as above.

It is useful to know that the projective norm and the operator canonical norm coincide on the space $\lin(\Gamma(\D))\otimes X$.

\begin{proposition}\label{teo-L} 
Given $\gamma\in\lin(\Gamma(\D))\otimes X$, we have $\left\|\gamma\right\|=\pi(\gamma)$, where 
$$
\left\|\gamma\right\|=\sup\left\{\left|\gamma(f)\right|\colon f\in\widehat{\B}(\D,X^*),\ p_{\B}(f)\leq 1 \right\}
$$
and 
$$
\pi(\gamma)=\inf\left\{\sum_{i=1}^n\frac{\left|\lambda_i\right|}{1-\left|z_i\right|^2}\left\|x_i\right\|\colon \gamma=\sum_{i=1}^n\lambda_i\gamma_{z_i}\otimes x_i\right\}.
$$
\end{proposition}

\begin{proof}
Let $\gamma\in\lin(\Gamma(\D))\otimes X$ and let $\sum_{i=1}^n\lambda_i\gamma_{z_i}\otimes x_i$ be a representation of $\gamma$. Since $\gamma$ is linear and
$$
\left|\gamma(f)\right|=\left|\sum_{i=1}^n\lambda_i\left\langle f'(z_i),x_i\right\rangle\right|
\leq\sum_{i=1}^n\left|\lambda_i\right|\left\|f'(z_i)\right\|\left\|x_i\right\|
\leq p_{\B}(f)\sum_{i=1}^n\left|\lambda_i\right|\frac{\left\|x_i\right\|}{1-\left|z_i\right|^2}
$$
for all $f\in\widehat{\B}(\D,X^*)$, we deduce that $||\gamma||\leq\sum_{i=1}^n|\lambda_i|||x_i||/(1-|z_i|^2)$. Since this holds for each representation of $\gamma$, it follows that $||\gamma||\leq\pi(\gamma)$ and thus $\left\|\cdot\right\|\leq\pi$ on $\lin(\Gamma(\D))\otimes X$. 

To prove the reverse inequality, suppose by contradiction that $\left\|\mu\right\|<1<\pi(\mu)$ for some $\mu\in\lin(\Gamma(\D))\otimes X$. Denote $B=\{\gamma\in\lin(\Gamma(\D))\otimes X\colon \pi(\gamma)\leq 1\}$. Clearly, $B$ is a closed convex subset of $\lin(\Gamma(\D))\otimes_\pi X$. Applying the Hahn--Banach Separation Theorem to $B$ and $\{\mu\}$, we obtain a functional $\eta\in(\lin(\Gamma(\D))\otimes_\pi X)^*$ such that 
$$
1=\|\eta\|=\sup\{\mathrm{Re}(\eta(\gamma))\colon\gamma\in B\}<\mathrm{Re}(\eta(\mu)). 
$$
Define $F_\eta\colon\D\to X^*$ by 
$$
\langle F_\eta(z),x\rangle=\eta\left(\gamma_z\otimes x\right)\qquad (x\in X,\; z\in\D).
$$
We now show that $F_\eta$ is holomorphic. By \cite[Exercise 8.D]{Muj-86}, it suffices to prove that for each $x\in X$, the function $F_{\eta,x}\colon\D\to\mathbb{C}$ defined by 
$$
F_{\eta,x}(z)=\eta(\gamma_z\otimes x)\qquad (z\in \D) 
$$
is holomorphic. Let $a\in\D$. Since $\Gamma\colon\D\to\lin(\Gamma(\D))$ is holomorphic, there exists $D\Gamma(a)\in\mathcal{L}(\C,\lin(\Gamma(\D)))$ such that 
$$
\lim_{z\to a}\frac{\gamma_z-\gamma_a-D\Gamma(a)(z-a)}{\left|z-a\right|}=0. 
$$
Consider the function $T(a)\colon\C\to\mathbb{C}$ given by 
$$
T(a)(z)=\eta(D\Gamma(a)(z)\otimes x)\qquad\left(z\in\C\right). 
$$
Clearly, $T(a)\in\L(\C,\C)$ and since 
\begin{align*}
F_{\eta,x}(z)-F_{\eta,x}(a)-T(a)(z-a)
&=\eta(\gamma_z\otimes x)-\eta(\gamma_a\otimes x)-\eta(D\Gamma(a)(z-a)\otimes x) \\
&=\eta\left((\gamma_z-\gamma_a-D\Gamma(a)(z-a))\otimes x\right),
\end{align*}
it follows that 
$$
\lim_{z\to a}\frac{F_{\eta,x}(z)-F_{\eta,x}(a)-T(a)(z-a)}{\left|z-a\right|}=\lim_{z\to a}\eta\left(\frac{\gamma_z-\gamma_a-D\Gamma(a)(z-a)}{\left|z-a\right|}\otimes x\right)=0.
$$
Hence $F_{\eta,x}$ is holomorphic at $a$ with $DF_{\eta,x}(a)=T(a)$, as desired.

By \cite[Lemma 2.9]{JimRui-22}, there exists a mapping $f_\eta\in\H(\D,X^*)$ with $f_\eta(0)=0$ such that $f_\eta'=F_\eta$. Given $z\in\D$, we have  
$$
(1-|z|^2)\left|\left\langle f_\eta'(z),x\right\rangle\right|=(1-|z|^2)\left|\eta\left(\gamma_z\otimes x\right)\right|\leq(1-|z|^2)\left\|\eta\right\|\pi(\gamma_z\otimes x)=\left\|x\right\|
$$
for all $x\in X$, and thus $(1-|z|^2)\left\|f_\eta'(z)\right\|\leq 1$. Hence $f_\eta\in\widehat{\B}(\D,X^*)$ with $p_{\B}(f_\eta)\leq 1$. Moreover, $\gamma(f_\eta)=\eta(\gamma)$ for all $\gamma\in\lin(\Gamma(\D))\otimes X$. Therefore $\left\|\mu\right\|\geq|\mu(f_\eta)|\geq\mathrm{Re}(\mu(f_\eta))=\mathrm{Re}(\eta(\mu))$, so $\left\|\mu\right\|>1$, and this is a contradiction.
\end{proof}


\subsection{$p$-Chevet--Saphar Bloch norms}

The $p$-Chevet--Saphar norms $d_p$ on the tensor product of two Banach spaces $E\otimes F$ are well known (see, for example, \cite[Section 6.2]{Rya-02}). 

Our study of the duality of the spaces $\Pi^{\widehat{\B}}_p(\D,X^*)$ requires the introduction of the following Bloch versions of such norms defined now on $\lin(\Gamma(\D))\otimes X$. 

The \emph{$p$-Chevet--Saphar Bloch norms} $d^{\widehat{\B}}_p$ for $1\leq p\leq \infty$ are defined on a $X$-valued Bloch molecule $\gamma\in\lin(\Gamma(\D))\otimes X$ as 
\begin{align*}
d^{\widehat{\B}}_1(\gamma)&=\inf\left\{\left(\sup_{g\in B_{\widehat{\B}(\D)}}\left(\max_{1\leq i\leq n}\left|\lambda_i\right|\left|g'(z_i)\right|\right)\right)\left(\sum_{i=1}^n\left\|x_i\right\|\right)\right\},\\
d^{\widehat{\B}}_p(\gamma)&=\inf\left\{\left(\sup_{g\in B_{\widehat{\B}(\D)}}\left(\sum_{i=1}^n\left|\lambda_i\right|^{p^*}\left|g'(z_i)\right|^{p^*}\right)^{\frac{1}{p^*}}\right)\left(\sum_{i=1}^n\left\|x_i\right\|^p\right)^{\frac{1}{p}}\right\}\quad (1<p<\infty),\\
d^{\widehat{\B}}_\infty(\gamma)&=\inf\left\{\left(\sup_{g\in B_{\widehat{\B}(\D)}}\left(\sum_{i=1}^n\left|\lambda_i\right|\left|g'(z_i)\right|\right)\right)\left(\max_{1\leq i \leq n}\left\|x_i\right\|\right)\right\},
\end{align*}
where the infimum is taken over all such representations of $\gamma$ as $\sum_{i=1}^n\lambda_i\gamma_{z_i}\otimes x_i$. 

Motivated by the analogue concept on the tensor product space (see \cite[p. 127]{Rya-02}), we introduce the following.

\begin{definition}
Let $X$ be a complex Banach space. A norm $\alpha$ on $\lin(\Gamma(\D))\otimes X$ is said to be a \emph{Bloch reasonable crossnorm} if it has the following properties:
\begin{enumerate}
	\item[(i)] $\alpha(\gamma_z\otimes x)\leq \left\|\gamma_z\right\|\left\|x\right\|$ for all $z\in\D$ and $x\in X$,
	\item[(ii)] For every $g\in\widehat{\B}(\D)$ and $x^*\in X^*$, the linear functional $g\otimes x^*\colon\lin(\Gamma(\D))\otimes X\to \C$ defined by $(g\otimes x^*)(\gamma_z\otimes x)=g'(z)x^*(x)$ is bounded on $\lin(\Gamma(\D))\otimes_\alpha X$ with $\left\|g\otimes x^*\right\|\leq p_\B(g)\left\|x^*\right\|$.
\end{enumerate}
\end{definition}

\begin{theorem}\label{teo-che-norms}
$d^{\widehat{\B}}_p$ is a Bloch reasonable crossnorm on $\lin(\Gamma(\D))\otimes X$ for any $1\leq p\leq\infty$.
\end{theorem}

\begin{proof}
We will only prove it for $1<p<\infty$. The other cases follow similarly. 

Let $\gamma\in\lin(\Gamma(\D))\otimes X$ and let $\sum_{i=1}^n\lambda_i\gamma_{z_i}\otimes x_i$ be a representation of $\gamma$. Clearly, $d^{\widehat{\B}}_p(\gamma)\geq 0$. Given $\lambda\in\C$, since $\sum_{i=1}^n (\lambda\lambda_i)\gamma_{z_i}\otimes x_i$ is a representation of $\lambda\gamma$, we have 
\begin{align*}
d^{\widehat{\B}}_p(\lambda\gamma)
&\leq\left(\sup_{g\in B_{\widehat{\B}(\D)}}\left(\sum_{i=1}^n\left|\lambda\lambda_i\right|^{p^*}\left|g'(z_i)\right|^{p^*}\right)^{\frac{1}{p^*}}\right)\left(\sum_{i=1}^n\left\|x_i\right\|^p\right)^{\frac{1}{p}}\\
&=\left|\lambda\right|\left(\sup_{g\in B_{\widehat{\B}(\D)}}\left(\sum_{i=1}^n\left|\lambda_i\right|^{p^*}\left|g'(z_i)\right|^{p^*}\right)^{\frac{1}{p^*}}\right)\left(\sum_{i=1}^n\left\|x_i\right\|^p\right)^{\frac{1}{p}}.
\end{align*}
If $\lambda=0$, we obtain $d^{\widehat{\B}}_p(\lambda\gamma)=0=\left|\lambda\right|d^{\widehat{\B}}_p(\gamma)$. For $\lambda\neq 0$, since the preceding inequality holds for every representation of $\gamma$, we deduce that $d^{\widehat{\B}}_p(\lambda\gamma)\leq\left|\lambda\right|d^{\widehat{\B}}_p(\gamma)$. For the converse inequality, note that $d^{\widehat{\B}}_p(\gamma)=d^{\widehat{\B}}_p(\lambda^{-1}(\lambda\gamma))\leq |\lambda^{-1}|d^{\widehat{\B}}_p(\lambda\gamma)$ by using the proved inequality, thus $\left|\lambda\right|d^{\widehat{\B}}_p(\gamma)\leq d^{\widehat{\B}}_p(\lambda\gamma)$ and hence $d^{\widehat{\B}}_p(\lambda\gamma)=\left|\lambda\right|d^{\widehat{\B}}_p(\gamma)$.

We now prove the triangular inequality of $d^{\widehat{\B}}_p$. Let $\gamma_1,\gamma_2\in\lin(\Gamma(\D))\otimes X$ and let $\varepsilon>0$. If $\gamma_1=0$ or $\gamma_2=0$, there is nothing to prove. Assume $\gamma_1\neq 0\neq \gamma_2$. We can choose representations 
$$
\gamma_1=\sum_{i=1}^n\lambda_{1,i}\gamma_{z_{1,i}}\otimes x_{1,i},\qquad
\gamma_2=\sum_{i=1}^m\lambda_{2,i}\gamma_{z_{2,i}}\otimes x_{2,i},
$$
so that
$$
\left(\sup_{g\in B_{\widehat{\B}(\D)}}\left(\sum_{i=1}^n\left|\lambda_{1,i}\right|^{p^*}\left|g'(z_{1,i})\right|^{p^*}\right)^{\frac{1}{p^*}}\right)\left(\sum_{i=1}^n\left\|x_{1,i}\right\|^p\right)^{\frac{1}{p}}\leq d^{\widehat{\B}}_p(\gamma_1)+\varepsilon
$$
and 
$$
\left(\sup_{g\in B_{\widehat{\B}(\D)}}\left(\sum_{i=1}^m\left|\lambda_{2,i}\right|^{p^*}\left|g'(z_{2,i})\right|^{p^*}\right)^{\frac{1}{p^*}}\right)\left(\sum_{i=1}^m\left\|x_{2,i}\right\|^p\right)^{\frac{1}{p}}\leq d^{\widehat{\B}}_p(\gamma_2)+\varepsilon .
$$ 
Fix arbitrary $r,s\in\mathbb{R}^+$ and define 
\begin{align*}
\lambda_{3,i}\gamma_{z_{3,i}}&=\left\{\begin{array}{lll}
r^{-1}\lambda_{1,i}\gamma_{z_{1,i}}&\text{ if } i=1,\ldots,n,\\
s^{-1}\lambda_{2,i-n}\gamma_{z_{2,i-n}}&\text{ if } i=n+1,\ldots,n+m,
\end{array}\right.\\
x_{3,i}&=\left\{\begin{array}{lll}
rx_{1,i}&\text{ if } i=1,\ldots,n,\\
sx_{2,i-n}&\text{ if } i=n+1,\ldots,n+m.
\end{array}\right.
\end{align*}
It is clear that $\gamma_1+\gamma_2=\sum_{i=1}^{n+m}\lambda_{3,i}\gamma_{z_{3,i}}\otimes x_{3,i}$ and thus we have
$$
d^{\widehat{\B}}_p(\gamma_1+\gamma_2)
\leq \left(\sup_{g\in B_{\widehat{\B}(\D)}}\left(\sum_{i=1}^{n+m}\left|\lambda_{3,i}\right|^{p^*}\left|g'(z_{3,i})\right|^{p^*}\right)^{\frac{1}{p^*}}\right)\left(\sum_{i=1}^{n+m}\left\|x_{3,i}\right\|^p\right)^{\frac{1}{p}}.
$$
An easy verification gives  
\begin{gather*}
\left(\sup_{g\in B_{\widehat{\B}(\D)}}\left(\sum_{i=1}^{n+m}\left|\lambda_{3,i}\right|^{p^*}\left|g'(z_{3,i})\right|^{p^*}\right)^{\frac{1}{p^*}}\right)^{p^*}\\
\leq\left(r^{-1}\sup_{g\in B_{\widehat{\B}(\D)}}\left(\sum_{i=1}^{n}\left|\lambda_{1,i}\right|^{p^*}\left|g'(z_{1,i})\right|^{p^*}\right)^{\frac{1}{p^*}}\right)^{p^*}
+\left(s^{-1}\sup_{g\in B_{\widehat{\B}(\D)}}\left(\sum_{i=1}^{m}\left|\lambda_{2,i}\right|^{p^*}\left|g'(z_{2,i})\right|^{p^*}\right)^{\frac{1}{p^*}}\right)^{p^*}
\end{gather*}
and
$$
\sum_{i=1}^{n+m}\left\|x_{3,i}\right\|^p
=r^{p}\sum_{i=1}^{n}\left\|x_{1,i}\right\|^p+s^p\sum_{i=1}^{m}\left\|x_{2,i}\right\|^p.
$$
Using Young's Inequality, it follows that 
\begin{align*}
d^{\widehat{\B}}_p(\gamma_1+\gamma_2)&\leq\frac{1}{p^*}\left(\sup_{g\in B_{\widehat{\B}(\D)}}\left(\sum_{i=1}^{n+m}\left|\lambda_{3,i}\right|^{p^*}\left|g'(z_{3,i})\right|^{p^*}\right)^{\frac{1}{p^*}}\right)^{p^*}+\frac{1}{p}\sum_{i=1}^{n+m}\left\|x_{3,i}\right\|^p\\
&\leq\frac{r^{-p^*}}{p^*}\left(\sup_{g\in B_{\widehat{\B}(\D)}}\left(\sum_{i=1}^{n}\left|\lambda_{1,i}\right|^{p^*}\left|g'(z_{1,i})\right|^{p^*}\right)^{\frac{1}{p^*}}\right)^{p^*}+\frac{r^p}{p}\sum_{i=1}^{n}\left\|x_{1,i}\right\|^p\\
&+\frac{s^{-p^*}}{p^*}\left(\sup_{g\in B_{\widehat{\B}(\D)}}\left(\sum_{i=1}^{m}\left|\lambda_{2,i}\right|^{p^*}\left|g'(z_{2,i})\right|^{p^*}\right)^{\frac{1}{p^*}}\right)^{p^*}
+\frac{s^p}{p}\sum_{i=1}^{m}\left\|x_{2,i}\right\|^p.
\end{align*}
Since $r,s$ were arbitrary in $\R^+$, taking above
\begin{align*}
r&=(d^{\widehat{\B}}_p(\gamma_1)+\varepsilon)^{-\frac{1}{p^*}}\left(\sup_{g\in B_{\widehat{\B}(\D)}}\left(\sum_{i=1}^{n}\left|\lambda_{1,i}\right|^{p^*}\left|g'(z_{1,i})\right|^{p^*}\right)^{\frac{1}{p^*}}\right),\\
s&=(d^{\widehat{\B}}_p(\gamma_2)+\varepsilon)^{-\frac{1}{p^*}}\left(\sup_{g\in B_{\widehat{\B}(\D)}}\left(\sum_{i=1}^{m}\left|\lambda_{2,i}\right|^{p^*}\left|g'(z_{2,i})\right|^{p^*}\right)^{\frac{1}{p^*}}\right),
\end{align*}
we obtain that $d^{\widehat{\B}}_p(\gamma_1+\gamma_2)\leq d^{\widehat{\B}}_p(\gamma_1)+d^{\widehat{\B}}_p(\gamma_2)+2\varepsilon$, and thus $d^{\widehat{\B}}_p(\gamma_1+\gamma_2)\leq d^{\widehat{\B}}_p(\gamma_1)+d^{\widehat{\B}}_p(\gamma_2)$ by the arbitrariness of $\varepsilon$. Hence $d^{\widehat{\B}}_p$ is a seminorm. To prove that it is a norm, note first that 
\begin{align*}
\left|\sum_{i=1}^n \lambda_i h'(z_i)x^*(x_i)\right|
&\leq\sum_{i=1}^n\left|\lambda_i\right|\left|h'(z_i)\right|\left\|x_i\right\|\\
&\leq\left(\sum_{i=1}^{n}\left|\lambda_i\right|^{p^*}\left|h'(z_i)\right|^{p^*}\right)^{\frac{1}{p^*}}\left(\sum_{i=1}^{n}\left\|x_i\right\|^p\right)^{\frac{1}{p}}\\
&\leq\sup_{g\in B_{\widehat{\B}(\D)}}\left(\sum_{i=1}^{n}\left|\lambda_i\right|^{p^*}\left|g'(z_i)\right|^{p^*}\right)^{\frac{1}{p^*}}\left(\sum_{i=1}^{n}\left\|x_i\right\|^p\right)^{\frac{1}{p}},
\end{align*}
for any $h\in B_{\widehat{\B}(\D)}$ and $x^*\in B_{X^*}$, by applying H\"{o}lder's Inequality. Since the quantity $\left|\sum_{i=1}^n \lambda_i h'(z_i)x^*(x_i)\right|$ does not depend on the representation of $\gamma$ because 
$$
\sum_{i=1}^n \lambda_i h'(z_i)x^*(x_i)=\left(\sum_{i=1}^n\lambda_i\gamma_{z_i}\otimes x_i\right)(h\cdot x^*)=\gamma(h\cdot x^*),
$$
taking the infimum over all representations of $\gamma$ we deduce that   
$$
\left|\sum_{i=1}^n \lambda_i h'(z_i)x^*(x_i)\right|\leq d^{\widehat{\B}}_p(\gamma)
$$
for any $h\in B_{\widehat{\B}(\D)}$ and $x^*\in B_{X^*}$. Now, if $d^{\widehat{\B}}_p(\gamma)=0$, the preceding inequality yields 
$$
\left(\sum_{i=1}^n \lambda_i x^*(x_i)\gamma_{z_i}\right)(h)=\sum_{i=1}^n \lambda_i x^*(x_i) h'(z_i)=0
$$
for all $h\in B_{\widehat{\B}(\D)}$ and $x^*\in B_{X^*}$. For each $x^*\in B_{X^*}$, this implies that $\sum_{i=1}^n\lambda_i x^*(x_i)\gamma_{z_i}=0$, and since $\Gamma(\D)$ is a linearly independent subset of $\G(\D)$ by \cite[Remark 2.8]{JimRui-22}, it follows that $x^*(x_i)\lambda_i =0$ for all $i\in\{1,\ldots,n\}$, hence $\lambda_i=0$ for all $i\in\{1,\ldots,n\}$ since $B_{X^*}$ separates the points of $X$, and thus $\gamma=\sum_{i=1}^n\lambda_i\gamma_{z_i}\otimes x_i=0$. 

Finally, we will show that $d^{\widehat{\B}}_p$ is a Bloch reasonable crossnorm on $\lin(\Gamma(\D))\otimes X$. Firstly, given $z\in\D$ and $x\in X$, we have
$$
d^{\widehat{\B}}_p(\gamma_z\otimes x)\leq\left(\sup_{g\in B_{\widehat{\B}(\D)}}\left|g'(z)\right|^{p^*}\right)^{\frac{1}{p^*}}\left\|x\right\|\leq \frac{\left\|x\right\|}{1-|z|^2}=\left\|\gamma_z\right\|\left\|x\right\|.
$$
Secondly, given $g\in\widehat{\B}(\D)$ and $x^*\in X^*$, we have 
\begin{align*}
\left|(g\otimes x^*)(\gamma)\right|&=\left|\sum_{i=1}^n\lambda_i(g\otimes x^*)(\gamma_{z_i}\otimes x_i)\right|=\left|\sum_{i=1}^n\lambda_i g'(z_i)x^*(x_i)\right|\\
&\leq\sum_{i=1}^n\left|\lambda_i\right|\left|g'(z_i)\right|\left|x^*(x_i)\right|\leq p_\B(g)\left\|x^*\right\|\sum_{i=1}^n\frac{\left|\lambda_i\right|}{1-|z_i|^2}\left\|x_i\right\|\\
&=p_\B(g)\left\|x^*\right\|\sum_{i=1}^n\left|\lambda_i\right|\left|f'_{z_i}(z_i)\right|\left\|x_i\right\|\\
&\leq p_\B(g)\left\|x^*\right\|\left(\sum_{i=1}^n\left|\lambda_i\right|^{p^*}\left|f'_{z_i}(z_i)\right|^{p^*}\right)^{\frac{1}{p^*}}\left(\sum_{i=1}^n\left\|x_i\right\|^p\right)^{\frac{1}{p}}\\
&\leq p_\B(g)\left\|x^*\right\|\sup_{g\in B_{\widehat{\B}(\D)}}\left(\sum_{i=1}^n\left|\lambda_i\right|^{p^*}\left|g'(z_i)\right|^{p^*}\right)^{\frac{1}{p^*}}\left(\sum_{i=1}^n\left\|x_i\right\|^p\right)^{\frac{1}{p}}.
\end{align*}
Taking infimum over all the representations of $\gamma$, we deduce that $\left|(g\otimes x^*)(\gamma)\right|\leq p_\B(g)\left\|x^*\right\|d^{\widehat{\mathcal{B}}}_p(\gamma)$. Hence $g\otimes x^*\in (\lin(\Gamma(\D))\otimes_{d^{\widehat{\mathcal{B}}_p}} X)^*$ with $\left\|g\otimes x^*\right\|\leq p_\B(g)\left\|x^*\right\|$. 
\end{proof}

The next result shows that $d^{\widehat{\B}}_p$ can be computed using a simpler formula in the cases $p=1$ and $p=\infty$. In fact, the $1$-Chevet--Saphar Bloch norm is justly the projective norm.

\begin{proposition}\label{1-nuclear-proj}
For $\gamma\in\lin(\Gamma(\D))\otimes X$, we have
$$
d^{\widehat{\B}}_1(\gamma)=\inf\left\{\sum_{i=1}^n\frac{\left|\lambda_i\right|}{1-\left|z_i\right|^2}\left\|x_i\right\|\right\}
$$ 
and 
$$
d^{\widehat{\B}}_\infty(\gamma)=\inf\left\{\sup_{g\in B_{\widehat{\B}(\D)}}\left(\sum_{i=1}^n\left|\lambda_i\right|\left|g'(z_i)\right|\left\|x_i\right\|\right)\right\},
$$
where the infimum is taken over all such representations of $\gamma$ as $\sum_{i=1}^n\lambda_i\gamma_{z_i}\otimes x_i$. 
\end{proposition}

\begin{proof}
Let $\gamma\in\lin(\Gamma(\D))\otimes X$ and let $\sum_{i=1}^n\lambda_i\gamma_{z_i}\otimes x_i$ be a representation of $\gamma$. We have 
\begin{align*}
\pi(\gamma)
&\leq\sum_{i=1}^n\frac{\left|\lambda_i\right|}{1-\left|z_i\right|^2}\left\|x_i\right\|=\sum_{i=1}^n\left|\lambda_i\right|\left(\sup_{g\in B_{\widehat{\B}(\D)}}\left|g'(z_i)\right|\right)\left\|x_i\right\|\\
&\leq\sum_{i=1}^n\max_{1\leq i\leq n}\left(\left|\lambda_i\right|\sup_{g\in B_{\widehat{\B}(\D)}}\left|g'(z_i)\right|\right)\left\|x_i\right\|=\left(\max_{1\leq i\leq n}\left(\left|\lambda_i\right|\sup_{g\in B_{\widehat{\B}(\D)}}\left|g'(z_i)\right|\right)\right)\sum_{i=1}^n\left\|x_i\right\|\\
&=\sup_{g\in B_{\widehat{\B}(\D)}}\left(\max_{1\leq i\leq n}\left(\left|\lambda_i\right|\left|g'(z_i)\right|\right)\right)\sum_{i=1}^n\left\|x_i\right\|\\
\end{align*}
and therefore $\pi(\gamma)\leq d^{\widehat{\B}}_1(\gamma)$. Conversely, since $d^{\widehat{\B}}_1$ is a Bloch reasonable crossnorm, we have 
$$
d^{\widehat{\B}}_1(\gamma)
\leq \sum_{i=1}^n\left|\lambda_i\right|d^{\widehat{\B}}_1(\gamma_{z_i}\otimes x_i)
=\sum_{i=1}^n\left|\lambda_i\right|\left\|\gamma_{z_i}\right\|\left\|x_i\right\|
=\sum_{i=1}^n\frac{\left|\lambda_i\right|}{1-|z_i|^2}\left\|x_i\right\|,
$$
and thus $d^{\widehat{\B}}_1(\gamma)\leq\pi(\gamma)$. 

On the other hand, we have 
$$
\sup_{g\in B_{\widehat{\B}(\D)}}\left(\sum_{i=1}^n\left|\lambda_i\right|\left|g'(z_i)\right|\left\|x_i\right\|\right)
\leq \left(\max_{1\leq i\leq n}\left\|x_i\right\|\right)\sup_{g\in B_{\widehat{\B}(\D)}}\left(\sum_{i=1}^n\left|\lambda_i\right|\left|g'(z_i)\right|\right),
$$
and taking the infimum over all representations of $\gamma$ gives  
$$
\inf\left\{\sup_{g\in B_{\widehat{\B}(\D)}}\left(\sum_{i=1}^n\left|\lambda_i\right|\left|g'(z_i)\right|\left\|x_i\right\|\right)\colon \gamma=\sum_{i=1}^n\lambda_i\gamma_{z_i}\otimes x_i\right\}\leq d^{\widehat{\B}}_\infty(\gamma).
$$
Conversely, we can assume without loss of generality that $x_i\neq 0$ for all $i\in\{1,\ldots,n\}$ and since $\gamma=\sum_{i=1}^n\lambda_i\left\|x_i\right\|\gamma_{z_i}\otimes (x_i/\left\|x_i\right\|)$, we obtain 
$$
d^{\widehat{\B}}_\infty(\gamma)\leq\sup_{g\in B_{\widehat{\B}(\D)}}\left(\sum_{i=1}^n\left|\lambda_i\right|\left\|x_i\right\|\left|g'(z_i)\right|\right),
$$
and taking the infimum over all representations of $\gamma$, we conclude that 
$$
d^{\widehat{\B}}_\infty(\gamma)\leq\inf\left\{\sup_{g\in B_{\widehat{\B}(\D)}}\left(\sum_{i=1}^n\left|\lambda_i\right|\left|g'(z_i)\right|\left\|x_i\right\|\right)\colon \gamma=\sum_{i=1}^n\lambda_i\gamma_{z_i}\otimes x_i\right\}.
$$
\end{proof}


\subsection{Duality}\label{4}

Given $p\in [1,\infty]$, we will show that the dual of the space $\G(\D)\widehat{\otimes}_{d^{\widehat{\B}}_{p^*}} X$ can be canonically identified as the space of $p$-summing Bloch mappings from $\D$ to $X^*$.

\begin{theorem}\label{messi-10}\label{theo-dual-classic-00} 
Let $1\leq p\leq\infty$. Then $\Pi^{\widehat{\B}}_p(\D,X^*)$ is isometrically isomorphic to $(\G(\D)\widehat{\otimes}_{d^{\widehat{\B}}_{p^*}} X)^*$, via the mapping $\Lambda\colon\Pi^{\widehat{\B}}_p(\D,X^*)\to(\G(\D)\widehat{\otimes}_{d^{\widehat{\B}}_{p^*}} X)^*$ defined by 
$$
\Lambda(f)(\gamma)=\sum_{i=1}^n\lambda_i\left\langle f'(z_i),x_i\right\rangle 
$$
for $f\in\Pi^{\widehat{\B}}_p(\D,X^*)$ and $\gamma=\sum_{i=1}^n\lambda_i\gamma_{z_i}\otimes x_i\in\lin(\Gamma(\D))\otimes X$. Furthermore, its inverse comes given by 
$$
\left\langle \Lambda^{-1}(\varphi)(z),x\right\rangle=\left\langle\varphi,\gamma_z\otimes x\right\rangle 
$$
for $\varphi\in(\G(\D)\widehat{\otimes}_{d^{\widehat{\B}}_{p^*}} X)^*$, $z\in\D$ and $x\in X$.

Moreover, on the unit ball of $\Pi^{\widehat{\B}}_p(\D,X^*)$ the weak* topology coincides with the topology of pointwise $\sigma(X^*,X)$-convergence.
\end{theorem}

\begin{proof}
We prove it for $1<p<\infty$. The cases $p=1$ and $p=\infty$ follow similarly.

Let $f\in\Pi^{\widehat{\B}}_p(\D,X^*)$ and let $\Lambda_0(f)\colon\lin(\Gamma(\D))\otimes X\to\mathbb{C}$ be its associate linear functional given by 
$$
\Lambda_0(f)(\gamma)=\sum_{i=1}^n\lambda_i\left\langle f'(z_i),x_i\right\rangle 
$$
for $\gamma=\sum_{i=1}^n\lambda_i\gamma_{z_i}\otimes x_i\in\lin(\Gamma(\D))\otimes X$. Note that $\Lambda_0(f)\in(\lin(\Gamma(\D))\otimes_{d^{\widehat{\B}}_{p^*}} X)^*$ with 
$\left\|\Lambda_0(f)\right\|\leq \pi^{\B}_p(f)$ since  
\begin{align*}
\left|\Lambda_0(f)(\gamma)\right|&=\left|\sum_{i=1}^n\lambda_i\left\langle f'(z_i), x_i\right\rangle\right|\leq\sum_{i=1}^n\left|\lambda_i\right|\left\|f'(z_i)\right\|\left\|x_i\right\|\\
&\leq \left(\sum_{i=1}^n\left|\lambda_i\right|^p\left\|f'(z_i)\right\|^p\right)^{\frac{1}{p}}\left(\sum_{i=1}^n\left\|x_i\right\|^{p^*}\right)^{\frac{1}{p^*}}\\
&\leq \pi^{\B}_p(f)\sup_{g\in B_{\widehat{\B}(\D)}}\left(\sum_{i=1}^n\left|\lambda_i\right|^p\left|g'(z_i)\right|^p\right)^{\frac{1}{p}}\left(\sum_{i=1}^n\left\|x_i\right\|^{p^*}\right)^{\frac{1}{p^*}},
\end{align*}
and taking infimum over all the representations of $\gamma$, we deduce that $\left|\Lambda_0(f)(\gamma)\right|\leq \pi^{\B}_p(f)d^{\widehat{\B}}_{p^*}(\gamma)$. Since $\gamma$ was arbitrary, then $\Lambda_0(f)$ is continuous on $\lin(\Gamma(\D))\otimes_{d^{\widehat{\B}}_{p^*}} X$ with $\left\|\Lambda_0(f)\right\|\leq \pi^{\B}_p(f)$.

Since $\lin(\Gamma(\D))$ is a norm-dense linear subspace of $\G(\D)$ and $d^{\widehat{\B}}_{p^*}$ is a norm on $\G(\D)\otimes X$, then $\G(\D)\otimes X$ is a dense linear subspace of $\G(\D)\otimes_{d^{\widehat{\B}}_{p^*}} X$ and therefore also of its completion $\G(\D)\widehat{\otimes}_{d^{\widehat{\B}}_{p^*}} X$. Hence there is a unique continuous mapping $\Lambda(f)$ from $\G(\D)\widehat{\otimes}_{d^{\widehat{\B}}_{p^*}} X$ to $\mathbb{C}$ that extends $\Lambda_0(f)$. Further, $\Lambda(f)$ is linear and $\left\|\Lambda(f)\right\|=\left\|\Lambda_0(f)\right\|$.

Let $\Lambda\colon\Pi^{\widehat{\B}}_p(\D,X^*)\to(\G(\D)\widehat{\otimes}_{d^{\widehat{\B}}_{p^*}} X)^*$ be the map so defined. Since $\Lambda_0$ is a linear monomorphism from $\Pi^{\widehat{\B}}_p(\D,X^*)$ to $(\G(\D)\otimes X)^*$ by Corollary \ref{linearization}, it follows easily that $\Lambda$ is so. To prove that $\Lambda$ is a surjective isometry, let $\varphi\in(\G(\D)\widehat{\otimes}_{d^{\widehat{\B}}_{p^*}} X)^*$ and define $F_\varphi\colon\D\to X^*$ by 
$$
\left\langle F_\varphi(z),x\right\rangle=\varphi(\gamma_z\otimes x)\qquad\left(z\in \D,\; x\in X\right). 
$$
As in the proof of Proposition \ref{teo-L}, it is similarly proved that $F_\varphi\in\H(\D,X^*)$ and there exists a mapping $f_\varphi\in\widehat{\B}(\D,X^*)$ with $p_{\B}(f_\varphi)\leq\left\|\varphi\right\|$ such that $f_\varphi'=F_\varphi$.

We now prove that $f_\varphi\in\Pi^{\widehat{\B}}_p(\D,X^*)$. Fix $n\in\mathbb{N}$, $\lambda_1,\ldots,\lambda_n\in\mathbb{C}$ and $z_1,\ldots,z_n\in\D$. Let $\varepsilon>0$. For each $i\in\{1,\ldots,n\}$, there exists $x_i\in X$ with $\left\|x_i\right\|\leq 1+\varepsilon$ such that $\left\langle f_\varphi'(z_i),x_i\right\rangle=\left\|f_\varphi'(z_i)\right\|$. It is clear that the map $T\colon\C^n\to\C$, defined by 
$$
T(t_1,\ldots,t_n)=\sum_{i=1}^n t_i \lambda_i\left\|f_\varphi'(z_i)\right\|,\quad\forall (t_1,\ldots,t_n)\in\C^n,
$$
is linear and continuous on $(\C^n,||\cdot||_{p^*})$ with $\left\|T\right\|=\left(\sum_{i=1}^n\left|\lambda_i\right|^p\left\|f'_\varphi(z_i)\right\|^{p}\right)^{\frac{1}{p}}$. For any $(t_1,\ldots,t_n)\in\C^n$ with $||(t_1,\ldots,t_n)||_{p^*}\leq 1$, we have
\begin{align*}
\left|T(t_1,\ldots,t_n)\right|&=\left|\varphi\left(\sum_{i=1}^n t_i\lambda_i\gamma_{z_i}\otimes x_i\right)\right|
\leq \left\|\varphi\right\|d^{\widehat{\B}}_{p^*}\left(\sum_{i=1}^n\lambda_i\gamma_{z_i}\otimes t_ix_i\right)\\
&\leq\left\|\varphi\right\|\left(\sup_{g\in B_{\widehat{\B}(\D)}}\left(\sum_{i=1}^n\left|\lambda_i\right|^{p}\left|g'(z_i)\right|^{p}\right)^{\frac{1}{p}}\right)\left(\sum_{i=1}^n\left\|t_ix_i\right\|^{p^*}\right)^{\frac{1}{p^*}}\\
&\leq(1+\varepsilon)\left\|\varphi\right\|\sup_{g\in B_{\widehat{\B}(\D)}}\left(\sum_{i=1}^n\left|\lambda_i\right|^{p}\left|g'(z_i)\right|^{p}\right)^{\frac{1}{p}},
\end{align*}
therefore 
$$
\left(\sum_{i=1}^n\left|\lambda_i\right|^p\left\|f'_\varphi(z_i)\right\|^{p}\right)^{\frac{1}{p}}
\leq(1+\varepsilon)\left\|\varphi\right\|\sup_{g\in B_{\widehat{\B}(\D)}}\left(\sum_{i=1}^n\left|\lambda_i\right|^{p}\left|g'(z_i)\right|^{p}\right)^{\frac{1}{p}},
$$
and since $\varepsilon$ was arbitrary, we have  
$$
\left(\sum_{i=1}^n\left|\lambda_i\right|^p\left\|f'_\varphi(z_i)\right\|^{p}\right)^{\frac{1}{p}}
\leq\left\|\varphi\right\|\sup_{g\in B_{\widehat{\B}(\D)}}\left(\sum_{i=1}^n\left|\lambda_i\right|^{p}\left|g'(z_i)\right|^{p}\right)^{\frac{1}{p}},
$$
and we conclude that $f_\varphi\in\Pi^{\widehat{\B}}_p(\D,X^*)$ with $\pi^{\B}_p(f_\varphi)\leq\left\|\varphi\right\|$. 

Finally, for any $\gamma=\sum_{i=1}^n \lambda_i\gamma_{z_i}\otimes x_i\in\lin(\Gamma(\D))\otimes X$, we get 
$$
\Lambda(f_\varphi)(\gamma) =\sum_{i=1}^n\lambda_i\left\langle f'_\varphi(z_i),x_i\right\rangle =\sum_{i=1}^n\lambda_i\varphi(\gamma_{z_i}\otimes x_i) =\varphi\left(\sum_{i=1}^n\lambda_i\gamma_{z_i}\otimes x_i\right) =\varphi(\gamma). 
$$
Hence $\Lambda(f_\varphi)=\varphi$ on a dense subspace of $\G(\D)\widehat{\otimes}_{d^{\widehat{\B}}_{p^*}} X$ and, consequently, $\Lambda(f_\varphi)=\varphi$, which shows the last statement of the theorem. Moreover, $\pi^{\B}_p(f_\varphi)\leq\left\|\varphi\right\|=\left\|\Lambda(f_\varphi)\right\|$.

For the final assertion of the statement, let $(f_i)_{i\in I}$ be a net in $\Pi^{\widehat{\B}}_p(\D,X^*)$ and $f\in\Pi^{\widehat{\B}}_p(\D,X^*)$. Assume $(f_i)_{i\in I}\to f$ weak* in $\Pi^{\widehat{\B}}_p(\D,X^*)$, this means that $(\Lambda(f_i))_{i\in I}\to \Lambda(f)$ weak* in $(\G(\D)\widehat{\otimes}_{d^{\widehat{\B}}_{p^*}} X)^*$, that is, $(\Lambda(f_i)(\gamma))_{i\in I}\to \Lambda(f)(\gamma)$ for all $\gamma\in\G(\D)\widehat{\otimes}_{d^{\widehat{\B}}_{p^*}} X$. In particular, 
$$
(<f'_i(z),x>)_{i\in I}=(\Lambda(f_i)(\gamma_z\otimes x))_{i\in I}\to \Lambda(f)(\gamma_z\otimes x)=<f'(z),x>
$$
for every $z\in\D$ and $x\in X$. Given $z\in\D$ and $x\in X$, we have
\begin{align*}
\left|\left\langle f_i(z)-f(z),x\right\rangle\right|&=\left|\int_{[0,z]}\left\langle f'_i(w)-f'(w),x\right\rangle\ dw\right|\\
                        &\leq |z|\max\left\{\left|\left\langle f'_i(w)-f'(w),x\right\rangle\right|\colon w\in [0,z]\right\}\\
												&=|z|\left|\left\langle f'_i(w_z)-f'(w_z),x\right\rangle\right|
\end{align*}
for all $i\in I$ and some $w_z\in [0,z]$, and thus $(\left\langle f_i(z),x\right\rangle)_{i\in I}\to \left\langle f(z),x\right\rangle$. This tells us that $(f_i)_{i\in I}$ converges to $f$ in the topology of pointwise $\sigma(X^*,X)$-convergence. Hence the identity on $\Pi^{\widehat{\B}}_p(\D,X^*)$ is a continuous bijection from the weak* topology to the topology of pointwise $\sigma(X^*,X)$-convergence. On the unit ball, the first topology is compact and the second one is Hausdorff, and so they must coincide.
\end{proof}

In particular, in view of Theorem \ref{messi-10} and taking into account Propositions \ref{prop-1}, \ref{teo-L} and \ref{1-nuclear-proj}, we can identify the space $\widehat{\B}(\D,X^*)$ with the dual space of $\G(\D)\widehat{\otimes} X\subseteq\widehat{\B}(\D,X^*)^*$.

\begin{corollary}\label{theo-dual-classic} 
$\widehat{\B}(\D,X^*)$ is isometrically isomorphic to $(\G(\D)\widehat{\otimes} X)^*$, via the mapping $\Lambda\colon\widehat{\B}(\D,X^*)\to(\G(\D)\widehat{\otimes} X)^*$ given by 
$$
\Lambda(f)(\gamma)=\sum_{i=1}^n\lambda_i\left\langle f'(z_i),x_i\right\rangle 
$$
for $f\in\widehat{\B}(\D,X^*)$ and $\gamma=\sum_{i=1}^n\lambda_i\gamma_{z_i}\otimes x_i\in\G(\D)\otimes X$. Furthermore, its inverse 
is given by 
$$
\left\langle \Lambda^{-1}(\varphi)(z),x\right\rangle=\left\langle\varphi,\gamma_z\otimes x\right\rangle 
$$
for $\varphi\in(\G(\D)\widehat{\otimes} X)^*$, $z\in\D$ and $x\in X$. $\hfill\qed$
\end{corollary}


We conclude this paper with some open questions we hope researchers will take up. 
In Theorem \ref{Pietsch2}, note that if $f\in\Pi^{\widehat{\B}}_2(\D,X)$, then 
$$
\iota_X\circ f'=T\circ I_{\infty,2}\circ h'\colon \D \stackrel{h'}{\rightarrow}L_\infty(\mu)\stackrel{I_{\infty,2}}{\rightarrow}L_2(\mu)\stackrel{T}{\rightarrow}\ell_{\infty}(B_{X^*}).
$$
Hence $\iota_X\circ f'$ factors in this way through the Hilbert space $L_2(\mu)$. It would be interesting to introduce and study the class of Bloch mappings whose derivatives factor through a Hilbert space.

Motivated by the seminal paper of Farmer and Johnson \cite{FarJoh-09} that raised a similar question in the setting of Lipschitz $p$-summing mappings, what results about $p$-summing linear operators have analogues for $p$-summing Bloch mappings?




 
 
 


\end{document}